\documentclass[11pt]{amsart}

\usepackage{amsmath}
\usepackage{amssymb}
\usepackage{amsthm}
\usepackage{graphicx}
\usepackage{tikz}
\usepackage{xcolor}
\usepackage{hyperref}
\usepackage{comment}
\usepackage{subfigure}

\allowdisplaybreaks

\usepackage[margin=1.2in]{geometry}
\def\E{\mathbb{E}}

\newcommand{\bl}{\boldsymbol{\lambda}}

\newcommand{\dist}{d}
\newcommand{\N}{\mathbb{N}}

\newcommand{\cN}{\mathcal{N}}

\newcommand{\one}{\mathbf{1}}

\newcommand{\vtau}{\vec \tau}
\newcommand{\vtauw}{\vec {\tau}_w}
\newcommand{\vsig}{\vec \sigma}

\newcommand{\bY}{\mathbf{Y}}
\newcommand{\bz}{\mathbf{z}}
\newcommand{\bw}{\mathbf{w}}

\def\E{\mathbb{E}}
\def\R{\mathbb{R}}
\def\P{\mathbb{P}}
\def\C{\mathbb{C}}
\def\Z{\mathbb{Z}}

\def\eps{\varepsilon}

\def\cA{\mathcal{A}}

\def\cX{\mathbb{X}}

\def\bv{\mathbf v}
\def\bx{\mathbf x}

\def\brho{\boldsymbol {\rho}}

\def\bgam{\mathbf \gamma}
\def\vbgam{ \vec{\boldsymbol{\gamma}}}
\def\vbgamw{ \vec{\boldsymbol{\gamma}}_w}

\def\1{\mathbf{1}}

\def\lam {\lambda}
\def\Lam{\Lambda}
\def\tce{t_c + \eps}
\def\tce2{t_c + \frac{\eps}{2}}
\def\ER{Erd\H{o}s-R\'{e}nyi }

\def\bv{\mathbf{v}}

\def\bX{\mathbf{X}}
\def\by{\mathbf{y}}

\newtheorem*{theorem*}{Theorem}
\newtheorem{theorem}{Theorem}
\newtheorem{lemma}[theorem]{Lemma}
\newtheorem{cor}[theorem]{Corollary}
\newtheorem{defn}[theorem]{Definition}
\newtheorem*{defn*}{Definition}
\newtheorem{prop}[theorem]{Proposition}
\newtheorem*{prop*}{Proposition}

\newtheorem*{conj*}{Conjecture}
\newtheorem{claim}[theorem]{Claim}
\newtheorem{question}{Question}

\newtheorem*{fact*}{Fact}
\newtheorem{fact}[theorem]{Fact}
\newtheorem{assumption}[theorem]{Assumption}

\theoremstyle{remark}
\newtheorem{remark}[theorem]{Remark}
\newtheorem*{remark*}{Remark}

\begin{document}
	\title[Connective constants and uniqueness of Gibbs measures]{Potential-weighted connective constants\\ and uniqueness of Gibbs measures}
	\author{Marcus Michelen}
	\author{Will Perkins}
	\address{Department of Mathematics, Statistics, and Computer Science\\ University of Illinois at Chicago}
	\email{michelen.math@gmail.com \\ math@willperkins.org}
	\date{\today}
	
	\begin{abstract}
	We define  a potential-weighted connective constant that measures the effective strength of a repulsive pair potential of a Gibbs point process modulated by the geometry of the underlying space.  We then show that this definition leads to improved bounds for Gibbs uniqueness for all non-trivial repulsive pair potentials on $\R^d$ and other metric measure spaces.  We do this by constructing a tree-branching collection of densities associated to the point process that captures the interplay between the potential and the geometry of the space.  When the activity is small as a function of the potential-weighted connective constant this  object exhibits an infinite volume uniqueness property.  On the other hand, we show that our uniqueness bound can be tight for certain spaces: the same infinite volume object exhibits non-uniqueness for activities above our bound in the case when the underlying space has the geometry of a  tree.
	\end{abstract}
	
	\maketitle

\section{Introduction}

Let $\phi$ be a symmetric function from $\R^d \times \R^d \to (-\infty, \infty]$. 
A finite-volume Gibbs point process acting via the pair potential $\phi$ on a bounded region $\Lam \subset \R^d$ at activity $\lam \ge 0$ is the probability measure $\nu$ on  finite point sets in $\Lam$ with density $e^{-H( \cdot)}$ against the Poisson process of intensity $\lam$ on $\Lam$, where 
\begin{align*}
H(x_1, \dots , x_k) = \sum_{1 \le i < j \le k} \phi (x_i,x_j)   \,.
\end{align*}
Infinite-volume Gibbs point processes on $\R^d$ can be defined via the Dobrushin--Lanford--Ruelle (DLR) or Georgii--Nguyen--Zessin (GNZ) equations (see, e.g.,~\cite{ruelle1999statistical,dereudre2019introduction,jansen2019cluster}).

  In statistical mechanics, the central questions about Gibbs point processes (or simply Gibbs measures) are about phase transitions.  Phase transitions can be defined in terms of the existence and uniqueness of  infinite-volume Gibbs measures  for a given potential $\phi$ and activity $\lam$.   General results say that under mild conditions on the potential $\phi$ there always exists such an infinite-volume  measure on $\R^d$.  The question of uniqueness or non-uniqueness is much more difficult.  In fact, though it is widely believed that there is non-uniqueness for a large class of potentials when $\lam$ is large enough, there is no known proof of non-uniqueness for the class of finite-range, rotationally symmetric pair potentials (including the widely-studied hard sphere model discussed below).  On the other hand,   when  the interaction $\phi$ is weak enough and the activity $\lam$ is small enough, general results tell us that there is a unique Gibbs measure.  Finding better and better conditions for uniqueness is thus a central topic in classical statistical physics, not only because of the connection to phase transitions but also because many techniques for proving uniqueness also give additional information such as correlation decay, mixing properties of Markov chains, or convergence of expansions for thermodynamic quantities.

One way to characterize the strength of a potential is through its \textit{temperedness constant}.  
A potential $\phi$ is {tempered} if 
\begin{equation*}
C_{\phi} = \sup_{v \in \R^d} \int_{\R^d} | 1- e^{-\phi(x,v)}  |\, dx  <\infty  \,,
\end{equation*}
and we call $C_{\phi}$ the temperedness constant.  

  Some condition on $\phi$ is also needed to ensure thermodynamic behavior.   A potential $\phi$ is \textit{stable} if there exists $B \ge 0$ so that for all $k\ge 2$ and all $x_1, \dots, x_k \in \R^d$, 
\[ \sum_{1 \le i < j \le k} \phi(x_i,x_j) \ge -B k \, .\]
A potential is \textit{repulsive} if $\phi(x,y) \ge 0$ for all $x,y$.  Every repulsive potential is stable with stability constant $B=0$. 

The most general criteria for uniqueness of the infinite volume Gibbs measure are in terms of $C_{\phi}$ and $B$. 
The classic result of Penrose~\cite{penrose1963convergence} and Ruelle~\cite{ruelle1963correlation} is that for $\lam < \frac{1}{ e^{2B+1}  C_{\phi}}$ the Gibbs measure is unique.  For the case of repulsive potentials,  the corresponding bound of $\frac{1}{e C_{\phi}}$ was obtained earlier by Groeneveld~\cite{groeneveld1962two}.
For repulsive potentials Meeron~\cite{meeron1970bounds} improved the  bound for uniqueness  to $1/C_{\phi}$, and the current authors further improved this to $e/C_{\phi}$~\cite{michelen2020analyticity}.  

The above bounds all depend only on the strength of the interaction as captured by the temperedness constant $C_{\phi}$ (and the stability constant $B$ in the general case), and not on the geometry of the underlying space $\R^d$ or the interaction between the potential $\phi$ and the geometry.  Nevertheless, in some cases geometric information has been used to improve bounds.  For instance, improved bounds on the radius of convergence of the cluster expansion for hard spheres~\cite{fernandez2007analyticity} and general repulsive potentials~\cite{jansen2019cluster,nguyen2020convergence} have been obtained in terms of multidimensional integrals.  The method of disagreement percolation has been used to link properties of continuum percolation  to uniqueness of Gibbs measures~\cite{christoph2019disagreement,hofer2019disagreement,betsch2021uniqueness}; this method is particularly effective in low-dimensional Euclidean space. 

Here we find  a general approach to combine   information about the potential $\phi$ and the geometry of the underlying space by defining a new notion of a \textit{potential-weighted connective constant}, inspired by the notion of the self-avoiding-walk connective constant of a lattice (see e.g.~\cite{hammersley1954poor,madras2013self}) that has found algorithmic and probabilistic applications in studying the discrete hard-core model~\cite{sinclair2013spatial,sinclair2017spatial}.   The potential-weighted connective constant $\Delta_{\phi}$ is the free energy of a  continuum  polymer model with an energy function and a step distribution that both depend on $\phi$. The constant $\Delta_\phi$ is always at most $C_{\phi}$, with a strict inequality  for any non-trivial potential $\phi$  on $\R^d$. 

This definition allows us to obtain the best known bounds for  uniqueness of Gibbs measures defined via repulsive pair potentials on $\R^d$ (and in  greater generality).     Our main result (Theorem~\ref{thmCCuniqueness} below) is uniqueness of the infinite-volume Gibbs measure for $\lam < e/ \Delta_{\phi}$. 

 We also explore the quality of this bound by constructing an \textit{infinite-depth tree recursion}, a collection of densities indexed by an infinitely branching tree, that combines information from the potential and the geometry of the underlying space (either $\R^d$ or some other metric measure space).  When the underlying space is a pure tree (a $0$-hyperbolic space for which $\Delta_{\phi} = C_{\phi}$) we show the bound $e/\Delta_{\phi}$ is tight: there is a uniqueness/non-uniqueness phase transition at $\lam = e/\Delta_{\phi}$, and so any further improvement to our bound will need to incorporate information about the underlying space beyond that captured by $\Delta_{\phi}$.

In the next section we formally define the potential-weighted connective constant and the general setting in which we  work, then  state our main results.  

\subsection{Potential weighted connective constants and absence of phase transition}

We start by considering a generalization of the Gibbs point processes defined above by allowing more general spaces than $\R^d$ (see~\cite{jansen2019cluster} for several fundamental results in such generality). 

Let $\cX$ be a complete, separable, metric space equipped with a  metric $d(\cdot, \cdot)$, and a locally-finite reference measure $\mu$ on the Borel sets of $\cX$ with respect to  $d$\footnote{We use `$d$' for both the dimension of $\R^d$ and for a metric but the meaning will be clear from the context.} .  Let $\phi : \cX \times \cX \to (-\infty, \infty]$ be a measurable symmetric function\footnote{Since $\phi$ is measurable, we know that for $\mu$ almost-all $v$ the function $x \mapsto \phi(v,x)$ is measurable.  By removing this possible null-set on which slices are not measurable, we may assume without loss of generality that $\phi(v,\bullet)$ is measurable for all $v \in \cX$.}. 
We extend the definition of the temperedness constant $C_\phi$ to be \begin{equation}\label{eqCphi}
	C_\phi := \sup_{v \in \cX} \int_{\cX} |1 - e^{-\phi(x,v)}| \,d\mu(x)
	\end{equation}
For $x \in \cX$, and locally finite $X \subset \cX$, let 
\[ H_x(X) := \sum_{y \in X} \phi(x,y) \, . \]

A Gibbs point process (or Gibbs measure) acting via the pair potential $\phi $ on  $\cX$ at  activity  $\lam$    is a probability measure $\nu$ on locally finite points sets of $\cX$ that satisfies the  GNZ equations:
\begin{equation}
\label{eqGNZ}
\E_{\nu}  \left [\sum_{x \in \cX} F(x,\mathbf X )  \right]  = \lam \int_{\cX}  \E_{\nu}   \left [     F(x, \mathbf X \cup \{x\})  e^{- H_x(\mathbf X )}  \right] d \mu(x)  
\end{equation} 
for all measurable non-negative test functions $F$, where we write $\bX$ for a sample from $\nu$.
 The existence of a Gibbs measure for a given $\cX$, $\phi$, and $\lam$ holds under the conditions considered here~\cite[Appendix B]{jansen2019cluster}.

We make one non-degeneracy assumption on the space $(\cX,\mu,d)$.  For each $x \in \cX$, consider the map $\pi_x:\cX \to \R_{\geq 0 }$ defined by $\pi_x(y) = d(y,x)$.  Let $\mu_x$ denote the pushfoward measure defined by $\mu_x = \mu \circ \pi_x^{-1}$.  In other words, for a Borel set $U \subset \R_{\geq 0}$, define $\mu_x(U) = \mu(\{y \in \cX : d(y,x) \in U\} )$.  We will assume the following.
	\begin{assumption}\label{assumption:continuity}
		For each $x \in \cX$, the measure $\mu_x$ on $[0,\infty)$ is absolutely continuous with respect to Lebesgue measure.
	\end{assumption} 
Assumption~\ref{assumption:continuity} is a kind of full-dimensionality assumption: one consequence is that thin spherical shells have small measure. For instance, on $\R^d$ it forbids a reference measure that gives positive measure to lower dimensional subspaces.   If $\cX = \R^d$, $\mu$ is absolutely continuous with respect to Lebesgue measure, and $d$ is the standard metric---or any equivalent metric---then $(\cX, \mu,d)$ satisfies Assumption~\ref{assumption:continuity}.
	
Now given such a space $(\cX, \mu,d)$ and a potential $\phi$ we can defined the potential-weighted connective constant.  Since the potential $\phi$ may take the value of $+\infty$, we adhere to the convention that $0 \cdot +\infty = 0$ in such instances.  We also use the convention that an empty sum is equal to $0$.
\begin{defn}
For a repulsive potential $\phi$,  and for each natural number $k$, define	
\begin{equation}\label{eq:Vk-def}
	V_k := \sup_{v_0 \in \cX} \int_{\cX^k} \prod_{j = 1}^k\left(\exp\left(- \sum_{i=0}^{j-2} \one_{d(v_j,v_i) < d(v_i,v_{i+1})} \phi(v_j, v_i) \right)\cdot\left(1 - e^{-\phi(v_j, v_{j-1})} \right)   \right)\,d\mu^k(\bv) \,.
	\end{equation}
The potential-weighted connective constant $\Delta_\phi$ is 
\begin{equation}
\label{eqDeltaphi}
\Delta_{\phi} := \lim_{k \to \infty} V_k^{1/k} = \inf_{k \geq 1} V_k^{1/k}\,.
\end{equation}
\end{defn}
The constant $\Delta_{\phi}$ is the exponential of the free energy of a polymer chain with with energy and step distribution defined depending on $d$ and $\phi$
(see e.g.~\cite{domb1972cluster,muthukumar1984perturbation,ioffe2010statistical} for more on polymer chains in the continuum and on lattices).
	
It follows  from the definition and the assumption that $\phi$ is repulsive that $\{V_k\}$ is submultiplicative, and so the limit exists and is equal to the infimum.  We provide the details in Section~\ref{secPotentialConstant}.  In particular, by bounding the first exponential term in the product by $1$, we have that $V_k \leq C_\phi^{k}$ and so  $\Delta_\phi \leq C_\phi$.  As in the case of the connective constant in the discrete setting, computing $\Delta_\phi$ exactly even for basic models appears to be intractable (see \cite{duminil-copin} for a notable exception, and more background on connective constants), although rigorous upper bounds may be proven by bounding $V_k$ for some specific $k$.

Our main result is that a Gibbs point process  defined by a repulsive pair potential exhibits uniqueness  for activities $\lam < e/\Delta_{\phi}$.   
\begin{theorem}
\label{thmCCuniqueness}
Let $\cX$ be a complete, separable, metric space equipped with a locally finite reference measure $\mu$ satisfying Assumption~\ref{assumption:continuity}.   Let $\phi$ be a repulsive, tempered potential  with potential-weighted connective constant $\Delta_{\phi}$ on $\cX$.  Then for $\lam< e/\Delta_{\phi}$, there is a unique Gibbs measure.
\end{theorem}

We also prove a  bound on complex evaluations of the partition function in finite volume.     Suppose that $\Lam \subset \cX$ with  $\mu ( \Lam) < \infty$.  Then the partition function of the model is
\begin{equation}
\label{eqPPpartition}
Z_{\Lam}(\lam) = 1+ \sum_{k \ge1 } \frac{\lam^k}{k!}   \int_{\Lam^k}  e^{-H(x_1, \dots ,x_k)}  \, d\mu(x_1) \cdots d\mu(x_k) \,.
\end{equation}
Convergence of the cluster expansion (e.g.~\cite{groeneveld1962two,penrose1963convergence,ruelle1963correlation}) implies a uniform bound on the magnitude of the finite volume (complex)  pressure, $| \log Z_{\Lam}(\lam)|/ \mu(\Lam)$, for $\lam$ in a disk around the origin in $\mathbb C$.  However, non-physical singularities of $\log Z$ on the negative real axis limit the applicability of the cluster expansion.  In~\cite{michelen2020analyticity} the current authors proved a uniform bound on the finite volume pressure for $\lam$ in a complex neighborhood  of $[0, \lam]$ for $\lam < e/C_{\phi}$.  Here we improve this by replacing the temperedness constant by the potential-weighted connective constant. 
\begin{theorem}
\label{thmZboundCC}
Let $\cX$ be a complete, separable, metric space equipped with a  locally finite reference measure $\mu$ satisfying Assumption~\ref{assumption:continuity}. Let $\phi$ be a repulsive, tempered potential  with potential-weighted connective constant $\Delta_{\phi}$.  Then for $\lam_0< e/\Delta_{\phi}$,  there is some simply connected open set $D \subset \mathbb C$ containing the interval $[0,\lam_0]$ and some $C>0$ so that for every $\lam \in D$ and every $\Lam \subset \cX$ with $\mu(\Lam) <\infty$, we have
\begin{equation}
| \log Z_{\Lam}(\lam) | \le C  \mu (\Lam) \, .
\end{equation}
\end{theorem}

Using this theorem we can deduce analyticity of the infinite volume pressure for repulsive point processes (with translation invariant potentials) on $\R^d$.  The infinite volume pressure is $p(\lam) = \lim_{n \to \infty}  \frac{1}{n} \log Z_{\Lam_n} (\lam)$, where $\Lam_n$ is the box of volume $n$ in $\R^d$.  
\begin{cor}
\label{corAnalytic}
Consider a Gibbs point process with a translation-invariant, tempered, repulsive potential $\phi$ on $\R^d$.  For activities $0 \le \lam < e/ \Delta_{\phi}$ the infinite volume pressure is analytic.  
\end{cor}
Corollary~\ref{corAnalytic} holds in greater generality than $\R^d$; what is required in addition to Theorem~\ref{thmZboundCC} is simply the existence of the limit of finite volume pressures.

\subsubsection{Example: the hard sphere model}
The hard sphere model (e.g.~\cite{alder1957phase,lowen2000fun,metropolis1953equation}) is defined by setting $\phi(x) = +\infty$ if $\| x\| <r$ and $0$ otherwise, for some $r>0$.  This potential forbids configurations of points in which  a pair of points is within distance $r$; in other words, valid configurations are sets of centers of sphere packings of spheres of radius $r/2$.  The hard sphere model is a model of a gas;  in dimensions two and three it is expected to exhibit a phase transition in the infinite volume limit~\cite{alder1957phase,bernard2011two}.

Let $v_{d,r}$ be the volume of the ball of radius $r$ in $\R^d$.  Then $C_{\phi} = v_{d,r}$ and so Groeneveld's  bound for uniqueness and analyticity via convergence of the cluster expansion is $1/(e v_{d,r})$~\cite{groeneveld1962two}, while Meeron's bound is $1/v_{d,r}$~\cite{meeron1970bounds}.  The  bound of Fern\'andez, Procacci, and Scoppola on cluster expansion convergence  in dimension $2$ is $.5107/ v_{2,r}$~\cite{fernandez2007analyticity}.   Hofer-Temmel~\cite{christoph2019disagreement}  proves a bound on uniqueness  of $2.1866/ v_{2,r}$ using disagreement percolation (note that the bound stated in~\cite{christoph2019disagreement} is too small by a factor $4$).  Helmuth, Perkins, and Petti proved uniqueness (and strong spatial mixing) in  dimension $d \ge 2$ for $\lam < 2/ v_{d,r}$~\cite{helmuth2020correlation}.  The current authors proved uniqueness and analyticity for $\lam < e/ v_{d,r}$~\cite{michelen2020analyticity}.   Theorem~\ref{thmCCuniqueness} and Corollary~\ref{corAnalytic} improve all of these bounds; in particular, we provide an explicit bound on $\Delta_\phi$ in the case of hard-spheres showing uniqueness for $\lam < e/(v_{d,r}(1 - 1/8^{d+1}))$ in Lemma~\ref{lem:HS-dimD}.  We calculate a better bound on the improvement in the case of dimension $2$ by calculating $V_2$ exactly.
\begin{cor}
\label{corHS2d} 
The hard sphere model on $\R^2$ exhibits uniqueness and analyticity for 
$$\lam < \frac{ e}{ \sqrt{ \frac{1}{2} + \frac{3 \sqrt{3}}{8 \pi}}     v_{2,r} } \approx  \frac{3.233}{v_{2,r}} \,.$$ 
\end{cor}
Note that this bound is approximately $6.33$ times larger than the best-known lower bound on the radius of convergence of the cluster expansion from~\cite{fernandez2007analyticity} and an improvement of a factor approximately $1.1895$ over the bound in~\cite{michelen2020analyticity}.  In Section~\ref{secPotentialConstant} we discuss possible further explicit improvements by obtaining better estimates on $\Delta_{\phi}$ for hard spheres and other potentials.

\subsection{Infinite-depth tree recursions}
\label{secIntroTree}

To prove the main theorem on uniqueness of Gibbs measures, in Section~\ref{secTrees} we construct a new object, an infinite  collection of point process densities structured as an uncountably branching tree of countable depth.

Recall that the \textit{density} of a Gibbs point process $\nu$ on $\cX$ is the function $\rho_\nu: \cX \to [0,\infty)$ so that for bounded, measurable $A \subset \cX$, 
\[ \int_A \rho_\nu(v) \, dv = \E_{\nu} | A \cap \mathbf X|    \]
 (see Section~\ref{secPrelim} for a formal definition).

Below in Proposition~\ref{pr:infinite-vol-recursion} we will prove the identity
	\begin{equation}
	\label{eqIdentityIntro}
	\rho_\nu(v) = \bl(v) \exp\left(-\int_{\cX} \rho_{\nu_{v \to w}}(w)(1 - e^{-\phi(v,w)})\,d\mu(w) \right)
	\end{equation}
	where $\nu_{v \to w}$ is an explicit Gibbs measure defined below.  Now for each $w \in \cX$ we can use~\eqref{eqIdentityIntro} again to write $\rho_{\nu_{v \to w}}(w)$ in the same form, but with densities of   different Gibbs measures in the integrand.   Repeating this inductively yields a  sequence of computations structured as a tree in which every node (corresponding to a density on the left-hand-side of~\eqref{eqIdentityIntro}) has uncountably many children (the densities that appear in the integral on the right-hand-side of~\eqref{eqIdentityIntro}).  We call this a tree recursion, and carrying this out to countably infinite depth yields an \textit{infinite-depth tree recursion} which we now define.

Let $\vbgam :   \bigcup_{k=1}^\infty \cX^k \to [0,1]$ be a measurable function.  We say $\vbgam$ is a \textit{damping function} if  $\vbgam(v) =1$ for all $v \in \cX$ and
\begin{equation} \label{eq:damping-sub-mult}
\vbgam(v_0,v_1 \dots , v_\ell) \le \vbgam(v_1,v_2, \dots, v_{\ell}) 
\end{equation}
for all $\ell \ge 1$ and all $(v_0,v_1, \dots, v_{\ell}) \in \cX^{\ell+1}$. 

Fix $\lam \ge 0$, and let $\pi : \bigcup_{k=1}^\infty \cX^k \to [0,\lam]$ be a measurable function and $\vbgam $ a damping function.   We say $\pi$ is an \textit{infinite-depth tree recursion} adapted to the pair $(\lam,  \vbgam)$ if for each $k \ge 0$ and each tuple $(v_0, v_1, \dots, v_k) \in \cX^{k+1}$, we have
\[  \pi (v_0, \dots, v_k) = \lam \cdot \vbgam(v_0,\dots,v_k) \exp \left(  -  \int_{\cX}   \pi(v_0, \dots, v_{k},w) (1- e^{- \phi(v_{k},w)}  ) \,  d \mu(w)    \right ) \,.  \]
The form of this equation arises from  applying \eqref{eqIdentityIntro}.

The damping function $\vbgam$  captures a notion of geometry.      In particular, for a given Gibbs measure $\nu$   we will construct an infinite-depth tree recursion $\pi$ so that $\pi(v) = \rho_{\nu}(v)$ for all $v \in \cX$.  The damping function  $\vbgamw$ in this construction is given explicitly and depends only on $\cX$, $\mu$, and $\phi$.  We set $\vbgamw(v)=1$ and $\vbgamw (u,v) =1$ for each $u,v \in \cX$, and for $k\ge 2$,
\begin{align*}
\vbgamw(v_0,v_1, \dots, v_k) = \exp\left(- \sum_{ i=0}^{k-2} \one_{d(v_k,v_i) < d(v_i,v_{i+1})} \phi(v_k, v_i) \right)  \,.
\end{align*}
As we will see in Section~\ref{secTrees}, the form of this damping function arises from recursively applying~\eqref{eqIdentityIntro}.

We will use uniqueness of infinite-depth tree recursions to prove uniqueness of Gibbs measures.
\begin{theorem}
\label{thmMainTree}
Fix a repulsive, tempered potential $\phi$ and a space $(\cX, \mu, d)$ satisfying Assumption~\ref{assumption:continuity}.   Suppose there is a unique infinite-depth tree recursion at activity $\lam$ for every damping function $\vbgam \leq \vbgam_w$.  Then there is a unique Gibbs measure on $(\cX,\mu, d)$ with potential $\phi$ and activity $\lam$. 
\end{theorem}

In other words, infinite-depth tree recursions can witness non-uniqueness: if there are multiple distinct Gibbs measures on $(\cX,\mu, d)$ with potential $\phi$ and activity $\lam$, then there is some $\vbgam \le \vbgam_w$ so that there exist multiple distinct  infinite depth tree recursions at activity $\lam$ with damping function $\vbgam$.

To prove Theorem~\ref{thmCCuniqueness} ,  we  associate a potential-weighted connective constant to a damping function $\vbgam$.  Define
\begin{equation}\label{eq:Vk-defTree}
	V_k(\vbgam) =  \sup_{v_0 \in \cX} \int_{\cX^k}  \prod_{j = 1}^k \vbgam(v_0, v_1, \dots , v_j)\left(1 - e^{-\phi(v_j, v_{j-1})} \right)   \,d\mu^k(\bv) \, ,
	\end{equation}
and define $\Delta_{\phi}(\vbgam) = \lim_{k \to \infty} V_k(\vbgam)^{1/k} = \inf_{k \ge 1}  V_k(\vbgam)^{1/k} $.  In particular, if we construct $\vbgamw$ from $(\cX,\mu, d)$ as above, then by definition $\Delta_{\phi}(\vbgamw) = \Delta_\phi$.  Theorem~\ref{thmCCuniqueness} then follows from Theorem~\ref{thmMainTree} and a uniqueness result for infinite-depth tree recursions.
\begin{theorem}
\label{thmUniqueTree}
Fix a repulsive, tempered potential $\phi$ and a space $(\cX, \mu,d)$ satisfying Assumption~\ref{assumption:continuity}. Let $\vbgam $ be a damping function. If $\lam < e/\Delta_{\phi} (\vbgam)$,  then there is at most one infinite-depth tree recursion $\pi$ adapted  to $(\lam, \vbgam)$.   
\end{theorem}

\subsection{Methods}

The inspiration for our methods comes from the algorithmic `correlation-decay method' for approximate counting and sampling in the discrete hard-core model due to Weitz~\cite{Weitz} and refinements based on the connective constant of a family of graphs due to Sinclair, Srivastava, {\v{S}}tefankovi{\v{c}}, and Yin~\cite{sinclair2013spatial,sinclair2017spatial}. Weitz gave an algorithm to approximate the marginal probability that a vertex $v$ is in the random independent set drawn according to the hard-core model on a graph $G$.  Using a recursion (and following a similar construction of Godsil~\cite{godsil1981matchings}), Weitz builds a `self-avoiding walk tree' or `computational tree' with the property that the probability the root of the tree is occupied in the hard-core model on the tree is exactly the probability $v$ is occupied in the hard-core model on $G$.  In general this tree may be exponentially large in the size of $G$, but when the activity $\lam$ is small enough, the tree exhibits strong spatial mixing and so by truncating the tree one may obtain a good approximation of the desired occupation probability.  For graphs of maximum degree $\Delta$, small enough means $\lam < \lam_c(\Delta) = \frac{ (\Delta-1)^{\Delta-1}}{(\Delta-2)^{\Delta}}$, the uniqueness threshold of the hard-core model on the infinite $\Delta$-regular tree.   Sinclair, Srivastava, {\v{S}}tefankovi{\v{c}}, and Yin showed that this bound is too pessimistic for families of graphs with some additional geometric properties.  They use the connective constant of a graph (the exponential growth rate of self-avoiding walks) to obtain better algorithmic bounds for families of graphs for which there is a substantial gap between the maximum degree minus $1$ and the connective constant; such families include low-dimensional lattices like $\Z^2$ and sequences of  sparse \ER random graphs for which the maximum degree is unbounded but the connective constant is bounded.  In particular, Sinclair, Srivastava, {\v{S}}tefankovi{\v{c}}, and Yin use this approach to obtain the best known bound on uniqueness of Gibbs measure for the hard-core model on $\Z^2$~\cite{sinclair2017spatial}.   The connective constant approach is specific to hard-core systems, namely the hard-core model and monomer-dimer models on a graph, and it remains an open problem to find a similar approach for more general spin systems on graphs (for the anti-ferromagnetic Ising model, for instance).

In~\cite{michelen2020analyticity}, the current authors proved a recursive identity for the density of a repulsive point process (a special case of Lemma~\ref{lemMainIdentitiy} below) inspired by one step of Weitz's recursion.  By analyzing contractive properties of this identity (and adapting ideas from the discrete setting in~\cite{peters2019conjecture}), we proved uniqueness of the infinite volume Gibbs measure and analyticity of the pressure for $\lam < e/C_{\phi}$.  

Here the potential-weighted connective constant allows us to achieve an improvement analogous to that of~\cite{sinclair2017spatial} for the hard-core model.  Our approach is not restricted to hard-core systems, but instead works for all repulsive pair potentials.  We leave as a future direction the question of adapting this definition back to the discrete setting.

\section{The potential-weighted connective constant}
\label{secPotentialConstant}
In this section we derive some properties of the potential-weighted connective constants $\Delta_{\phi}(\vbgam)$ and $\Delta_\phi$ and give an upper bound in some special cases.

	We first  show submultiplicativity of $V_k(\vbgam)$ as defined in~\eqref{eq:Vk-defTree}.  
	\begin{lemma}
		For any damping function $\vbgam$ and $k, \ell \in \N$ we have $V_{k + \ell}(\vbgam) \leq V_k(\vbgam) V_{\ell}(\vbgam)$.  Thus $\lim_{k \to \infty} V_k(\vbgam)^{1/k} = \inf_{k \geq 1} V_k(\vbgam)^{1/k}$.
	\end{lemma}
	\begin{proof}
		We note that for any tuple $(v_0,\ldots,v_{k + \ell})$ we have \begin{align*}\prod_{j = 1}^{k + \ell} \vbgam(v_0,\ldots,v_j)(1 - e^{-\phi(v_j, v_{j-1})}) &\leq \left(\prod_{j = 1}^{k}\vbgam(v_0,\ldots,v_j) (1 - e^{-\phi(v_j, v_{j-1})})\right) \\
		&\times \prod_{j = 1}^{\ell} \vbgam(v_k,\ldots,v_{k+j})(1 - e^{-\phi(v_{k+j}, v_{k+j-1})}) \,. \end{align*}
		Integrating over variables $v_{k+1},\ldots,v_{k+\ell}$ first followed by $v_1,\ldots,v_k$ shows shows $V_{k+\ell}(\vbgam) \leq V_k(\vbgam) V_{\ell}(\vbgam)$.  Fekete's lemma then shows $\lim_{k \to \infty} V_k(\vbgam)^{1/k} = \inf_{k \geq 1} V_k(\vbgam)^{1/k}$.
	\end{proof}

\subsection{Hard disks}

To understand the definition of $V_k$, and thus that of $\Delta_\phi$,  consider  the example of hard spheres in $\R^d$.  The potential $\phi$ is given by $\phi(x,y) = \infty$ for $\|x -y \|_2 < r$ and $\phi(x,y) = 0$ for $\|x - y\|_2 \ge r$, and so letting $d$ be the standard $\ell^2$ metric we have
$$V_k = \int_{(\R^d)^k} \prod_{j=1}^k\one\{\dist(x_j,x_{j-1}) < r\} \cdot \prod_{i = 0}^{j-2} \one\{ \dist(x_j,x_i) > \dist(x_i,x_{i+1})   \}\,d\bx\,. $$
In words, $V_k$ is the measure of tuples $(v_1,\ldots,v_k)$ where adjacent points are within $r$ of each other and points later in the tuple are forbidden from the disks centered at $x_i$ with boundary containing $x_{i+1}$.  Figure \ref{figure} shows such a tuple for $k = 5$.

\begin{figure}[h] \label{figure}
	\centering
	\subfigure{
		\includegraphics[width=1.15in]{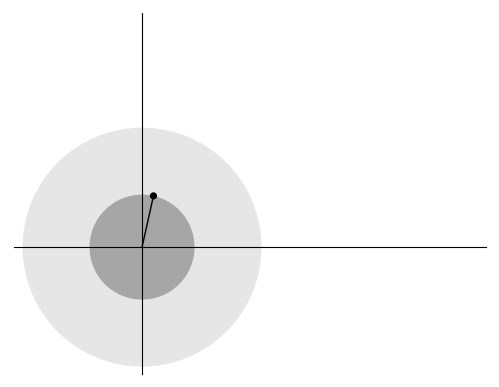} \hspace{-.1in}}
	\subfigure {
		\includegraphics[width=1.15in]{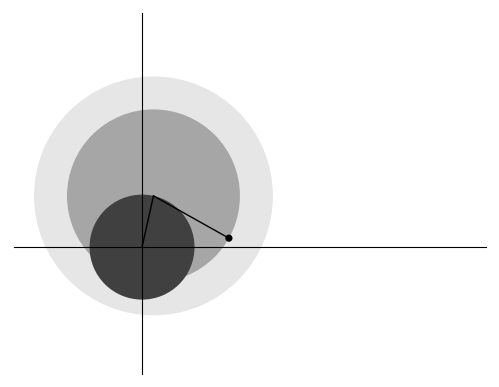} \hspace{-.1in}}
	\subfigure {
		\includegraphics[width=1.15in]{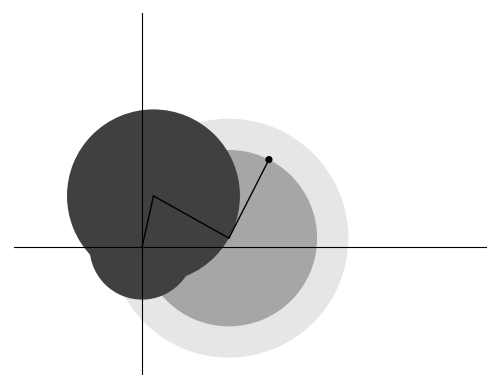} \hspace{-.1in}}
	\subfigure {
		\includegraphics[width=1.15in]{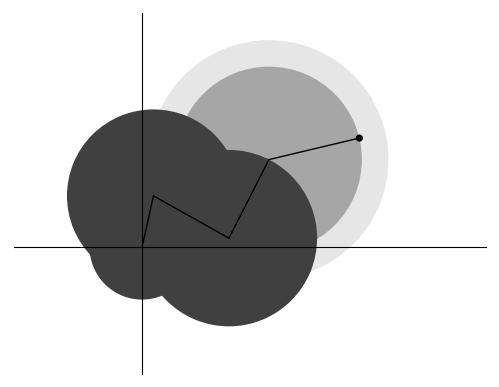} \hspace{-.1in}}
	\subfigure {
		\includegraphics[width=1.15in]{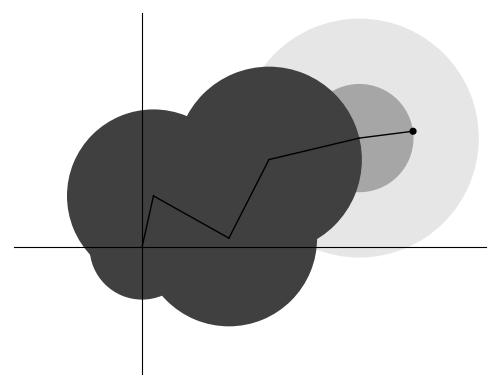} }
	\caption{A tuple counted in $V_5$ for the hard disk model.  At each stage, a new point is chosen uniformly at random from the light gray disk.  If it lies in the dark gray region, then the entire tuple is rejected.  If not, then we draw the medium gray disk centered at the previous point with the new point on the boundary, and add this new disk to the dark gray forbidden region.}
\end{figure}

We now give an upper bound on $\Delta_{\phi}$ for hard disks (hard spheres  in dimension $2$). 
By submultiplicativity we have $\Delta_\phi \leq \sqrt{V_2}$ and it suffices to prove an upper bound for 

$$V_2 = \int_{B^2}  \one_{\|w_1 + w_2\|_2 > \|w_1\|_2} \,dw_1\,dw_2\,,$$
where $B$ is the disk of radius $r$ (and volume $v_{2,r}$) around $0$. 

Using some basic planar geometry we can compute $V_2$ exactly.
\begin{lemma}
\label{lemHDv2}
	In dimension $2$ we have $V_2 =v_{2,r}^2 (\frac{1}{2} + \frac{3\sqrt{3}}{8 \pi})$, and so  the hard disk potential $\phi$ satisfies $\Delta_\phi \leq  \sqrt{ \frac{1}{2} + \frac{3 \sqrt{3}}{8\pi}} \cdot v_{2,r}$.
\end{lemma}
\begin{proof}
	The ratio $V_2 / v_{2,r}^2$ is constant with respect to $r$, and so we may assume for simplicity that $r = 1$.  		If $\|w_1\|_2 \leq 1/2$, note then that $\int_B \one_{\|w_1 + w_2\|_2 > \|w_1\|_2} \,dw_2= \pi( 1 - \|w_1\|_2^2).$
	When $s := \|w_1\|_2 > 1/2$, then the formula for the area of a lens gives \begin{align*}
	\int_B \one_{\|w_1 + w_2\|_2 > \|w_1\|_2} \,dw_2 &= \pi - \left(s^2 \arccos\left(1 - \frac{1}{2s^2} \right) + \arccos\left(\frac{1}{2s}\right) - \frac{1}{2}\sqrt{(2s-1)(2s+1)}\right)\,.
	\end{align*}
	
	We then have \begin{align*}
	V_2 &= 2\pi^2\int_0^{1/2} s(1 - s^2)\,ds \\
	&\qquad + 2\pi \int_{1/2}^1 s\left(\pi - \left(s^2 \arccos\left(1 - \frac{1}{2s^2} \right) + \arccos\left(\frac{1}{2s}\right) - \frac{1}{2}\sqrt{(2s-1)(2s+1)}\right)\right)\,ds \\
	&= \frac{1}{2}\pi^2 + \frac{3 \sqrt{3}}{8} \pi\,. 
	\end{align*}
	
	We thus have $V_2/v_{2,r}^2 = \frac{1}{2} + \frac{3 \sqrt{3}}{8\pi}$.
\end{proof}

Corollary~\ref{corHS2d} follows from Lemma~\ref{lemHDv2} and Theorems~\ref{thmCCuniqueness} and~\ref{thmZboundCC}.  

\begin{remark}
	We note that this approach leaves  open the possibility for further improvements to the bound given in Lemma~\ref{lemHDv2} by computational means; since $\Delta_\phi \leq V_k^{1/k}$ for each $k$, (rigorous) approximations for the terms $V_k$ provide upper bounds for $\Delta_\phi$.  By submultiplicativity, this bound gets tighter for larger and larger $k$. With this in mind, it appears that the bound $\Delta_\phi \leq V_2^{1/2}$ is somewhat slack for the hard disk model.  We do not take up the approach of computing better approximations rigorously although a Monte-Carlo simulation for $V_{20}$ suggests that in fact $\Delta_\phi \leq .62 \cdot C_\phi$.  
\end{remark}

	A similar approach gives a crude but non-trivial upper bound on $\Delta_\phi$ for both hard spheres and hard cubes (hard spheres in the $\ell^{\infty}$-metric) in dimension $d \geq2$.  
	\begin{lemma}\label{lem:HS-dimD}
		For hard-spheres or hard-cubes in dimension $d \geq 2$ we have 
		$$\Delta_\phi < (1 - (1/8)^{d+1})C_\phi\,.$$
	\end{lemma}
	\begin{proof}
		The potential is of the form $\phi(x,y) = +\infty$ for $\|x - y\|_{\ast} < r$ and $0$ otherwise, where for hard cubes $\| \cdot \|_{\ast}$ is the $\ell^\infty$-norm and for hard spheres it is the $\ell^2$-norm.  We write $\| \cdot \|$ for this norm, for simplicity.	As in Lemma \ref{lemHDv2}, we take $r = 1$, and we will bound  $V_2$ where we set $x_0 = 0$.  We have $$V_2 = \int_{B_1^2} \one_{\| w_1 + w_2 \| > \|w_1\|} \,dw_1\,dw_2$$
		where we choose the  $\ell^2$-norm for the metric $d$ in the case of hard spheres and $\ell^\infty$-norm in the case of hard cubes.  
		Define $U = \{\|w_1\| \in [1/4,1/2], \|w_1 + w_2\| \leq 1/4 \}$ and note that on $U$ we have $\|w_1 + w_2\| \leq \|w_1\|$.  Further, the volume of $U$ is $|U| = (|B_{1/2}| -  |B_{1/4}|)\cdot |B_{1/4}|$ and so $$\frac{V_2}{|B_1|^2} \leq 1 - \frac{|U|}{|B_1|^2} = 1 - \frac{1}{8^d} + \frac{1}{16^d}\,.$$
		This gives $\Delta_\phi / C_\phi \leq \sqrt{V_2 / |B_1|^2 } \leq 1 - 8^{-(d+1)}$ as claimed. 
	\end{proof}

Similar explicit bounds on $\Delta_{\phi}$ can be computed for other potentials by computing or bounding $V_k$.  The calculations above can easily be adapted to the case of the Strauss process~\cite{strauss1975model}, with potential $\phi(x,y) =  a \cdot \one_{ \| x- y \| \le r }$ for some $a>0$, a soft-interaction version of the hard sphere model.  We obtain uniqueness for the Strauss process on $\R^2$ for 
$$\lam < \frac{e}{C_\phi}\left(\frac{1}{2} + \frac{3 \sqrt{3}}{8\pi}  + e^{-a}\left(\frac{1}{2} - \frac{3 \sqrt{3}}{8\pi} \right)\right)^{-1/2}$$
 by computing $V_2$ exactly.

\subsection{Relation to curvature} 
For what spaces $\cX$ should we expect $\Delta_{\phi}$ to be a significant improvement over $C_{\phi}$?  A rough answer is that the lower the dimension and the less negatively curved the space, the greater the improvement that can be expected.

 If the underlying space $\cX$ is a Riemannian manifold,   the gap  between $\Delta_\phi$ and $C_{\phi}$ is related to the curvature of the manifold.  We illustrate this in the special case of a Riemannian manifold with constant curvature.  In this case, the proof of Lemma~\ref{lem:HS-dimD} goes through essentially unchanged.

\begin{fact} \label{fact:curv}
	Let $M$ be a $n$-dimensional Riemannian manifold with constant sectional curvature $\kappa$.  Suppose that $\phi$ is of finite range $R > 0$, i.e. $\phi(x,y) = 0$ if $d(x,y) > R$; suppose further that there is some $\eps > 0$ so that for all $(x,y)$ with $d(x,y) \leq \eps$ we have $\phi(x,y) = c \in (0,+\infty]$.  Then $\Delta_\phi \leq (1 - \delta)C_\phi$ where $\delta > 0$ depends on $\phi, n$ and $\kappa$. 
\end{fact}
\begin{proof}
	Define $B_p(r)$ to be the (geodesic) ball of radius $r$ centered at $p$; since $M$ is of constant sectional curvature, the volume of $B_r(p)$ is a function only of $n,r$ and $\kappa$ (see, e.g., \cite{petersen2006riemannian}).  For a given point $p \in M$, define $U = \{ (x,y) \in M^2 : \dist(x,p) \in [\eps/4,\eps/2], \dist(y,p) \leq \eps/4  \}$.  Then the contribution to $C_\phi^2$ given from the integral over $U$ is exactly $(|B_{\eps/2}| - |B_{\eps/4}|)\cdot |B_{\eps/4}| (1 - e^{-c})^2$.  However, the contribution to $V_2$ on $U$ is $e^{-c }(|B_{\eps/2}| - |B_{\eps/4}|)\cdot |B_{\eps/4}| (1 - e^{-c})^2$.   Since $\phi$ has radius $R$, $C_\phi \leq |B_R|$, and so we have that $V_2 \leq (1 - \delta) C_\phi^2$ for some $\delta$.   Taking square roots of both sides completes the proof.
\end{proof}

We note that as $\kappa$ tends to $-\infty$, then the volume of $U$ is a vanishingly small proportion of $|B_\eps|^2$, and so the $\delta$ in this bound tends to $0$ as $\kappa \to -\infty$. 

Intuitively, large \emph{negative} curvature appears to be the obstacle towards achieving a strict upper bound on $\Delta_\phi/C_\phi$.   We suspect that this is true in a more general setting than  Fact \ref{fact:curv} states.  In particular, it appears that an assumption of constant curvature is unnecessarily strong, and that simply a \emph{lower bound} on curvature is all that is needed for the connective constant to be strictly less than $C_\phi$; one piece of evidence towards this is the Bishop-Cheeger-Gromov comparison theorem, which states that if a complete manifold has Ricci curvature bounded below, then balls can only grow as fast (in volume) as in a corresponding hyperbolic space (see \cite{petersen2006riemannian} for more context and a precise statement).  We put forward a question in this direction.  

\begin{question}
	In the setting of a Riemannian manifold with where Ricci curvature is bounded below by $\kappa$, do we have that $\Delta_\phi \leq (1 - \delta)C_\phi$, where $\delta >0$ depends only on the dimension, $\kappa$, and $\phi$?
\end{question}

 \section{Gibbs measures, densities, and recursions}
 \label{secPrelim}

In this section we present some definitions and lemmas about Gibbs point processes and their accompanying density functions in finite and infinite volume.  Much of the background and many fundamental results about these processes can be found in~\cite{ruelle1999statistical,jansen2019cluster}.

We  then prove a recursive identity for the density of a point process (following the identity for finite-volume densities in~\cite{michelen2020analyticity}) which will be the crucial tool in our construction of tree recursions in Section~\ref{secTrees}.

\subsection{Point process preliminaries}

Fix   a complete, separable metric space $\cX$ equipped with a metric $d$ and a locally finite Borel measure $\mu$ satisfying Assumption~\ref{assumption:continuity}.

	We let $\mathcal{B}$ denote the Borel sets on $\cX$.  A \emph{locally-finite counting measure} is a measure $\nu$ on $\cX$  with $\nu(A) \in \N_0$ for all bounded $A \in \mathcal{B}$. 
 Let $\mathcal{N}$ denote the set of locally finite counting measures on $\cX$ respectively and let $\mathfrak{N}$ be the $\sigma$-algebra on $\mathcal{N}$ generated by the maps $\nu \mapsto \nu(A)$ for each $A \in \mathcal{B}$.   For a measurable set $\Lam \subset \cX$, let $\mathcal{N}(\Lam)$ denote the set of locally-finite counting measures on $\Lam$ and $\mathfrak{N}(\Lam)$ be the associated $\sigma$-algebra.  
	
	A point process is a random counting measure on $\cX$ that is measurable with respect to $\mathfrak{N}$.  Each instance of a random counting measure $\eta$ can be identified with a finite or countable  set of  points  that correspond to its atoms; more specifically, there is a set $ X= \{ x_1, x_2,\ldots \}$ so that $\eta = \sum_{x \in X} \delta_{x}$.  We will write $\mathbf X$ for the random set of points of a point process.
	
We  generalize slightly to the case of inhomogeneous activity functions. The generalization introduce some redundancy since the inhomogeneity could be incorporated into the reference measure $\mu$, but it will be convenient for our proofs.  An \textit{activity function} $\bl:\cX \to \R_{\geq 0}$ is a measurable function on $\cX$ with $\int_B \bl(x) d \mu(x) <\infty$ for every bounded $B \in \mathcal{B}$.  It is bounded by $\lam$ if $\bl(x) \le \lam$ for all $x \in \cX$.   
	
A Gibbs point process  on $(\cX, \mu, d)$ with potential $\phi$ and activity function $\bl$ is a probability measure $\nu$ on $\mathcal N$ satisfying the GNZ equations:
 	\begin{equation}
\label{eqGNZ2}
\E_{\nu}  \left [\sum_{x \in \bX} F(x,\mathbf X )  \right]  =  \int_{\cX} \bl(x)  \E_{\nu}   \left [     F(x, \mathbf X \cup \{x\})  e^{- H_x(\mathbf X )}  \right] d \mu(x)   
\end{equation} 
	for every measurable function $F :  \cX \times \mathcal N \to [0, \infty)$.

A useful fact about repulsive point processes is that they are stochastically dominated by Poisson processes.
\begin{lemma}[\cite{georgii1997stochastic}]\label{lem:Poisson-domination}
	Let $\bX$ be a Gibbs point process on $\cX$ associated to activity function $\bl$ and repulsive potential $\phi$.  Then $\bX$ is stochastically dominated by the Poisson process of intensity $\bl$ on $\cX$ in the sense that if $\bY$ is the Poisson process of intensity $\bl$, then there is a coupling of $(\bX,\bY)$ so that $\bX \subseteq \bY$.
\end{lemma}

\subsection{Densities}

Our main objects of study will be the  density functions of a Gibbs point process  (both the $1$-point and $k$-point densities).

Given a Gibbs measure $\nu$ on $\cX$ associated with activity function $\bl$ and potential $\phi$,   the \textit{density} $\rho_\nu: \cX \to [0, \infty)$ is defined by
\begin{equation}
\label{eq1Pt1}
\rho_{\nu} (v) =   \bl(v) \E_{\nu} e^{-H_v(\bX)}  \,.
\end{equation}
By \eqref{eqGNZ}, the density has the property that its integral over a region $B$ gives the expected number of points of the point process in $B$:
\[ \int_{B} \rho_{\nu} (v) d \mu(v) = \E_{\nu} | \bX \cap B|  \, . \] 

The \textit{k-point density function} $\rho_{\nu} : \cX^k \to [ 0, \infty)$ is
\begin{equation}
\label{eqkPt1}
\rho_{\nu} (v_1, \dots, v_k) =   e^{-H(v_1,\ldots,v_k) }  \cdot \prod_{j=1}^k \bl(v_j)\cdot \E_{\nu} e^{- \sum_{j=1}^k H_{v_j}(\bX)}  \,.
\end{equation}

Our method for proving uniqueness of the Gibbs measure will be proving uniqueness of the $k$-point density functions, and the following lemma~\cite{last2018lectures}. 
\begin{lemma}[\cite{last2018lectures}]
\label{lemRuelleBound}
Suppose there exists a constant $C>0$ so that 
\[ \rho_{\nu} (v_1, \dots, v_k)  \le C^k  \]
for all $v_1, \dots, v_k \in \cX$. Then the collection of $k$-point density functions $\rho_{\nu}(v_1,\dots,v_k)$, $k \ge 1$, $v_1, \dots, v_k \in \cX$, determine the Gibbs measure $\nu$. 
\end{lemma}
The bound $\rho_\nu(v_1,\ldots,v_k) \leq C^k$ is sometimes called the \emph{Ruelle bound}, and Lemma \ref{lemRuelleBound} may be viewed as a point process version of Carleman's condition from probability theory.

It will be useful for us to modify an activity function $\bl$ by decreasing it at a set of points by pointwise multiplication by a function $f: \cX \to [0,1]$.  We let $\bl \cdot f$ denote the function defined by $(\bl\cdot f)(v) = \bl(v) f(v)$.  The next lemma says that if this modification is bounded in a certain sense, then we obtain a well-defined Gibbs measure as a result of the modification.
\begin{lemma}\label{fact:Gibbs}
	Let $\nu$ be a Gibbs measure associated to an activity function $\bl$ and repulsive potential $\phi$.  Let $f:\cX \to [0,1]$ be a measurable function so that $\E_\nu \prod_{x \in \bX}f(x) > 0$.  Then the measure $\nu_f$ defined by 
	\[ \nu_f(\mathcal{E}) \propto \int_{\mathcal{E}} \prod_{x \in \bX}f(x)\,d\nu(\bX) \]
	is a Gibbs measure with potential $\phi$ and activity function $ \bl_f= \bl \cdot f$.
\end{lemma}
\begin{proof}
	We need to show that $\nu_f$ satisfies the GNZ equations with activity function $\bl_f $.  Set $\Xi = \E_\nu \prod_{x \in \bX} f(x)$ to be the constant of proportionality in the definition of $\nu_f$. For a given test function $F$, define $G$ by $G(x,\bX) = \left(\prod_{y \in \bX} f(y)\right) F(x,\bX)$.  Write  
	\begin{align*}
	\E_{\nu_f}\left[\sum_{x \in \bX} F(x,\bX) \right] &= \Xi^{-1} \E_{\nu}\left[\prod_{y \in \bX}f(y)\sum_{x \in \bX} F(x,\bX) \right] = \Xi^{-1} \E_{\nu}\left[\sum_{x \in \bX} G(x,\bX) \right] \\
	&= \Xi^{-1} \int_{\cX} \bl(x) \E_{\nu} \left[ G(x,\bX \cup \{x\}) e^{-H_x(\bX)} \right]\,d\mu(x) \\
	&= \int_{\cX} \bl(x)f(x) \E_{\nu_f}\left[ F(x,\bX \cup \{x\}) e^{-H_x(\bX)} \right]\,d\mu(x)
	\end{align*}
	where on the second line we applied the GNZ equation to the Gibbs measure $\nu$ with function $G$.  This shows $\nu_f$ satisfies the GNZ equations with activity $\bl_f $, and is thus a Gibbs measure for that activity.
\end{proof}

\subsection{Integral identities}

	The following is an extension of~\cite[Theorem 8]{michelen2020analyticity} which gave a recursive integral identity for the density of a finite-volume Gibbs point process with a repulsive potential.
	\begin{prop}\label{pr:infinite-vol-recursion}
	Let $\nu$ be a Gibbs measure associated to activity $\bl$ and repulsive potential $\phi$, and suppose $\mu$ satisfies Assumption \ref{assumption:continuity}.  For any point $v \in \cX$ we have \begin{equation}
	\rho_\nu(v) = \bl(v) \exp\left(-\int_{\cX} \rho_{\nu_{v \to w}}(w)(1 - e^{-\phi(v,w)})\,d\mu(w) \right)
	\end{equation}
	where $\nu_{v \to w}$ is the Gibbs measure defined by $$\nu_{v \to w}(\mathcal{E}) \propto \int_{\mathcal{E}} \exp\left(- \sum_{x \in \bX} \phi(x,v) \one_{d(x,v) < d(v,w)} \right) \,d\nu(\bX) \,.$$
\end{prop}

In other words, $\nu_{v \to w}$ is the Gibbs measure obtained from Lemma~\ref{fact:Gibbs} with $f(x) = \exp(-\phi(x,v)\one_{d(x,v) < d(w,v)}  )$.  We show in Lemma \ref{lem:lb-density} that $f$ satisfies the hypothesis of Lemma \ref{fact:Gibbs} and so $\nu_{v \to w}$ is well-defined.
\begin{proof}
	Let $B_t(v)$ denote the ball of radius $t$ centered at $v$. 
	For each $t > 0$, let $\nu_t$ be the probability measure defined by $$\nu_t(\mathcal{E}) \propto \int_\mathcal{E} \exp\left(-\sum_{x \in \bX \cap B_t(v)} \phi(v,x)\right)\,d\nu(\bX)\,.$$ By Lemma~\ref{fact:Gibbs}, $\nu_t$ is a Gibbs measure associated to the activity $x \mapsto \bl(x) e^{- \phi(v,x) \one_{\{x \in B_t(v) \}}}$.  For a sequence $0 = t_0 < t_1 < \ldots < t_M < t_{M+1} = +\infty$ we have the telescoping product \begin{equation}\label{eq:telescoping-product}
	\E_{\nu} e^{-H_v(\bX)} = \prod_{j = 0}^M \E_{\nu_{t_j}} e^{-\sum_{x \in \bX \cap (B_{t_{j+1}}(v) \setminus B_{t_j}(v ))}  \phi(x,v)}\,. 
	\end{equation}
	Let $\eps \in (0,1)$, and $R > 0$; we will ultimately take $\eps \to 0^+$ followed by $R \to \infty$.   Choose $M$ and the sequence $\{t_j\}$ so that $t_M = R$ and $\mu(B_{t_{j+1}}(v) \setminus B_{t_j}(v)) \leq \eps$ for $j < M$.  Note that this is possible by Assumption~\ref{assumption:continuity}.  For simplicity, write $S_j := B_{t_j}(v)$ and $\nu_j := \nu_{t_j}$.
	
	\begin{claim}\label{claim:log-expansion}
		For $j \in \{0,\ldots,M-1\}$ we have \begin{align*}\log \E_{\nu_{j}} e^{-\sum_{x \in \bX \cap (S_{j+1}\setminus S_j)}  \phi(x,v)} &= - \int_{S_{j+1}\setminus S_j} (1 - e^{-\phi(v,w)}) \rho_{\nu_j}(w)\,d\mu(w) \\
		&\quad + O\left(\lambda^2 \eps \int_{S_{j+1}\setminus S_j}(1 - e^{-\phi(v,w)}) \,d\mu(w) \right)\end{align*}
	\end{claim}
	\begin{proof}[Proof of Claim \ref{claim:log-expansion}]
			Write $A := S_{j+1} \setminus S_j$.  By Poisson domination (Lemma~\ref{lem:Poisson-domination}), bound $$\left|1 -  \E_{\nu_{j}}e^{-\sum_{x \in \bX \cap A} \phi(v,x)}\right| \leq \lambda \int_{A}(1 - e^{-\phi(v,w)}) \,d\mu(w) \,.$$
			Thus we have \begin{align}\label{eq:log-expand}
			\log \E_{\nu_{j}} e^{-\sum_{x \in \bX \cap A}  \phi(x,v)} = - \E_{\nu_{j}}\left(1 - e^{- \sum_{x \in \bX \cap A} \phi(v,x)}\right)  +  O\left(\lambda^2 \eps \int_{A}(1 - e^{-\phi(v,w)}) \,d\mu(w) \right)  \,.\end{align}
			Note that \begin{equation}\label{eq:linearize}
			\left|1 - e^{- \sum_{x \in \bX \cap A} \phi(v,x)} - \sum_{x \in \bX \cap A}\left(1 - e^{-\phi(v,x)} \right)\right| \leq \one_{|\bX\cap A| \geq 2} \sum_{x \in \bX \cap A}(1 - e^{-\phi(v,x)})\end{equation} 
			and by Poisson domination  we have 
			\begin{equation}\label{eq:tail-compute}
			\E_{\nu_{j}} \left[\one_{|\bX\cap A| \geq 2} \sum_{x \in \bX \cap A}(1 - e^{-\phi(v,x)})\right] =O\left(\lambda^2 \eps \int_{A}(1 - e^{-\phi(v,w)}) \,d\mu(w) \right)\,.
			\end{equation}
			Combining lines \eqref{eq:log-expand}, \eqref{eq:linearize} and \eqref{eq:tail-compute} shows $$\log \E_{\nu_{j}}\exp\left(-\sum_{x \in \bX \cap A} \phi(v,x)\right) = -\E_{\nu_{j}} \sum_{x \in \bX \cap A}(1 - e^{-\phi(v,x)}) + O\left(\lambda^2 \eps \int_{A}(1 - e^{-\phi(v,w)}) \,d\mu(w) \right)\,.$$
			The GNZ equation \eqref{eqGNZ} shows $$\E_{\nu_{j}}  \sum_{x \in \bX \cap A}(1 - e^{-\phi(v,x)}) = \int_{A} (1 - e^{-\phi(v,w)}) \rho_{\nu_{j}}(w) \,d\mu(w)$$
			completing the proof.
	\end{proof}
	
	Claim \ref{claim:log-expansion} handles all terms in the product \eqref{eq:telescoping-product} aside from the last one; Poisson domination and temperedness shows that the final term tends to $1$ as $R \to \infty$ and so Claim \ref{claim:log-expansion} shows
	\begin{equation}\label{eq:claim-app}
	\E_\nu e^{-H_v(\bX)} = \exp\left( - \sum_{j = 0}^{M-1} \int_{S_{j+1} \setminus S_j} (1 - e^{-\phi(v,x)}) \rho_{\nu_j}(w) \,d\mu(w) + o(1) + O(\lambda^2 C_\phi \eps)   \right)
	\end{equation} 
	where the $o(1)$ term is as $R \to \infty$.  Keeping $R$ fixed and sending $\eps \to 0^+$, the bounded convergence theorem shows \begin{equation}\label{eq:eps-lim}
	\sum_{j = 0}^{M-1} \int_{S_{j+1} \setminus S_j} (1 - e^{-\phi(v,x)}) \rho_{\nu_j}(w) \,d\mu(w) \xrightarrow{\eps \to 0^+} \int_{B_R(v) }(1 - e^{-\phi(v,x)}) \rho_{\nu_{v \to w}}(w)\,d\mu(w)\,.
	\end{equation}
	Taking $\eps \to 0^+$ followed by $R \to \infty$ and combining lines \eqref{eq:claim-app} and \eqref{eq:eps-lim} completes the proof.	
\end{proof}

We can decompose the $k$-point density function of a Gibbs measure $\nu$ into a product of $1$-point densities for altered Gibbs measures.  In order to do this, we will need to check the condition of Lemma~\ref{fact:Gibbs}, and so we show that densities are non-zero in the support of $\bl$.

\begin{lemma}\label{lem:lb-density}
	Let $\nu$ be an infinite volume Gibbs measure associated to an activity function $\bl$ and repulsive potential $\phi$.  Then if $\bl \leq \lambda$, then for all $v_1,\ldots,v_k$ we have $$\E_\nu e^{-\sum_{j = 1}^k H_{v_j}(\bX)} \geq \frac{1}{2}e^{-6 \lambda k C_\phi}\,.$$
\end{lemma}
\begin{proof}
	We will again use the Poisson domination guaranteed by Lemma~\ref{lem:Poisson-domination} and lower bound the expectation by the corresponding expectation for a Poisson process; let $\bY$ denote the Poisson process on $\cX$ with intensity $\lambda$ against the measure $\mu$.  Define $S_j:= \{x : \phi(v_j,x) \geq 1\}$ and $S = \cup S_j$.  Note that since $\bY$ is a Poisson process, the processes $\bY \cap S$ and $\bY \cap S^c$ are independent.  By temperedness, we have $\mu(S) \leq 2k C_\phi$ and so with probability $e^{-2 \lambda k C_\phi}$ we have $\bY \cap S = \emptyset$.  Note that $$\sum_{j} \E H_{v_j}(\bY \cap S^c) \leq \sum_{j} \lambda \int_{S_j^c} \phi(v_j,x) \,d\mu(x) \leq \sum_j 2\lambda \int_{S_j^c}(1 - e^{-\phi(v_j,x)})\,d\mu(x) \leq 2 \lambda kC_\phi\,.$$
	By Markov's inequality, this implies $$\P_{\bY \cap S^c}(e^{-\sum_{j = 1}^k H_{v_j}(\bY \cap S^c) } \leq e^{-4\lambda k C_\phi} ) \leq \frac{1}{2}\,.$$
	We may then bound $$\E_\bY e^{- \sum_{j = 1}^k H_{v_j}(\bY)} \geq e^{-2 \lambda k C_\phi} \E_{\bY \cap S^c} e^{- \sum_{j=1}^k H_{v_j}(\bY \cap S^c) } \geq \frac{1}{2}e^{-2 \lambda k C_\phi } e^{- 4 \lambda k C_\phi }\,. $$
\end{proof}

We now show that $k$-point densities of infinite volume Gibbs measures may be written as a product of $1$-point densities of Gibbs measures. 

\begin{lemma}\label{lemkPointProd}
	Let $\nu$ be a Gibbs measure associated to activity function $\bl$ and repulsive potential $\phi$.  For a tuple $(v_1,\ldots,v_k) \in \cX^k$, define the functions $f_1,\ldots,f_k$ by $f_j(w) =  e^{-\sum_{i = 1}^{j-1} \phi(v_i,w)}$, and let  $\nu_1, \dots, \nu_k$ and $\bl_1,\ldots, \bl_k$ be the Gibbs measures and activities derived from $(\nu,\bl)$ and the functions $f_1,\ldots,f_k$ via Lemma~\ref{fact:Gibbs}.  Then
	$$\rho_\nu(v_1,\ldots,v_k) = \prod_{j = 1}^k \rho_{\nu_j}(v_j) \,.$$
	
\end{lemma}
\begin{proof}
	Define the measure $\nu_j$ via  $$\nu_j(\mathcal{E}) \propto \int_\mathcal{E} e^{- \sum_{i = 1}^{j-1} H_{v_i}(\bX) } \,d\nu(\bX)$$ and note that $\nu_j$ is a Gibbs measure associated to activity $\bl_j$ by Lemma~\ref{fact:Gibbs}; we note that the hypotheses of Lemma~\ref{fact:Gibbs} are met due to Lemma~\ref{lem:lb-density}.  
	Write \begin{align*}\prod_{j = 1}^k \rho_{\nu_j}(v_j) &= \prod_{j = 1}^k \left(\bl(v_j)e^{-\sum_{i = 1}^{j-1}\phi(v_j,v_i)}  \frac{\E_{\nu} e^{- \sum_{i = 1}^{j} H_{v_i}(\bX) }  }{\E_{\nu} e^{- \sum_{i = 1}^{j - 1} H_{v_i}(\bX) } }   \right) \\
	&= e^{-H(v_1,\ldots,v_k)}\left(\prod_{j = 1}^k \bl(v_j) \right) \E_\nu e^{- \sum_{i = 1}^k H_{v_i}(\bX)} \\
	&= \rho_\nu(v_1,\ldots,v_k)\end{align*}
	where the second equality is due to the fact that the product telescopes. 
\end{proof}

	\section{Tree recursions}
	\label{secTrees}
	
	In this section we construct   finite- and infinite-depth tree-structured computations of the density of a Gibbs point process, using the identity of Proposition~\ref{pr:infinite-vol-recursion} and modeled after Weitz's computational tree in the discrete setting~\cite{Weitz} (and the earlier construction of Godsil for matchings in a graph~\cite{godsil1981matchings}). 
	
Throughout this section fix a space $(\cX,\mu, d)$ satisfying Assumption~\ref{assumption:continuity} and a tempered, repulsive potential $\phi$.

As mentioned in Section~\ref{secIntroTree},  we can use the integral identity of Proposition~\ref{pr:infinite-vol-recursion} recursively to construct a  tree-structured sequence of computations for the density of a point process, $\rho_{\nu}$.  We will show that analyzing this recursive  computation lets us deduce uniqueness properties of the model, provided $\lambda$ is small enough as a function of $\phi$ and $(\cX,\mu,d)$. 

\subsection{Finite-depth tree recursions}
	
We start by  constructing  finite-depth tree recursions. These will be defined in terms of the space $(\cX, \mu,d)$, the potential $\phi$, an activity function $\bl$, and two other objects: a damping function and a boundary condition.

   For $k \ge 1$,  we define a \textit{depth-k damping function} to be a measurable function  $\vbgam: \bigcup_{j=1}^{k} \cX^j \to [0,1]$ with the following properties:
\begin{enumerate}
\item $\vbgam(v) = 1$ for all $v \in \cX$ 
\item  For all $1 \leq \ell < k$, and all $(v_0, \dots v_{\ell})\in \cX^{\ell+1}$, 
\begin{equation}
\label{eqVBsubmult}
\vbgam(v_0, \dots , v_{\ell} )\le  \vbgam(v_1, \dots , v_{\ell} )  \, . 
\end{equation}
\end{enumerate}

A \textit{depth-k boundary condition} is  a bounded  measurable 
function $\vtau : \cX^{k+1} \to [0,\infty)$.

For $k\ge 0$, the \textit{depth-k tree recursion} with activity function $\bl$, damping function $\vbgam$, and boundary condition $\vtau$ is the function $\pi_{\bl,\vtau,\vbgam} : \bigcup_{j=1}^{k+1} \cX^j$ defined by:
\begin{equation}
\label{eqDepthk1}
\pi_{\bl,\vtau,\vbgam}(v_0,\ldots,v_k) = \vtau(v_0,\ldots,v_{k}) 
\end{equation}
and 
\begin{align} 
\pi_{\bl,\vtau,\vbgam}(v_0,\dots, v_{j} ) &= \bl(v_{j}) \cdot \vbgam(v_0, \dots , v_{j}) \nonumber \\
&\quad\times \exp \left(   - \int_{\cX} (1-e^{-\phi( v_j,w )}) \pi_{\bl,\vtau,\vbgam}(v_0,\ldots,v_{j},w) \, d\mu(w) \right)\label{eqDepthk2}
\end{align}
for $j=0, \dots, k-1$.
	
	We interpret the recursion as follows: consider a depth-$k$ tree with nodes indexed by tuples of points from $\cX$ and a root with index $v_0 \in \cX$.  The children of node $x_{v_0, \dots, v_j}$ are the nodes $x_{v_0, \dots, v_j, w}$ for $w \in \cX$.  Following Proposition~\ref{pr:infinite-vol-recursion}, we assign densities to nodes in this tree by integrating over its children.  For the nodes at depth $k$ (tuples of size $k+1$)  we specify densities via the boundary conditions $\vec \tau$.  Note that a depth-$0$ recursion employs no damping function: the output is simply the boundary condition.  
	
We first observe that the recursion makes sense:  all the functions being integrated are integrable. 
	\begin{lemma}\label{lemTreeMeasurable}
		Fix $k\ge 1$,  a $\lam$-bounded activity function $\bl$, and depth-$k$ boundary condition $\vtau$ and damping function $\vbgam$ as above.  For each $0\le j \le k-1$, the function $(v_0, \dots v_j) \mapsto \pi_{\bl,\vtau,\vbgam}(v_0, \dots, v_j)$ is measurable and bounded by $\lambda$.  As a consequence, for almost all $(v_0,\ldots,v_j)$, the function $x \mapsto \pi_{\bl,\vtau,\vbgam}(v_0, \dots, v_j, \bullet)$ is measurable.   
	\end{lemma}
\begin{proof}
		The second statement follows immediately from the first, and so it is sufficient to show measurability of $(v_0, \dots v_j) \mapsto \pi_{\bl,\vtau,\vbgam}(v_0, \dots, v_j)$.  We prove so by induction on $k - j$ and note that when $k - j = 0$ this follows from the assumption that $\vtau$ is measurable.  We suppose now that $(v_0, \dots v_j) \mapsto \pi_{\bl,\vtau,\vbgam}(v_0, \dots, v_j)$ is measurable and seek to prove $(v_0, \dots v_{j-1}) \mapsto \pi_{\bl,\vtau,\vbgam}(v_0, \dots, v_{j-1})$ is measurable.   We note that $\vbgam$ and $\bl$  are measurable and that the product of measurable functions is measurable.  By the inductive hypothesis, the function $\pi_{\bl,\vtau,\vbgam}(v_0, \dots v_{j-1},\bullet)$ is measurable for almost all $(v_0,\ldots,v_{j-1})$ and thus we have measurability of $(v_0,\ldots,v_{j-1}) \mapsto \int_{\cX}(1 - e^{-\phi(x,v_{k-1})}) \vtau(v_0,\ldots,v_{k-1},x)\,d\mu(x)$.  Postcomposing by the continuous function $x \mapsto e^{-x}$ preserves measurability, thus completing the proof.  
	\end{proof}
	As Lemma~\ref{lemTreeMeasurable} allows us define tree recursions for $\mu^{j+1}$-almost-all $(v_0,\ldots,v_j)$, we say that two tree recursions are equal if they are equal on almost-all tuples of each size.

 We  note also that there is a recursive structure:  within a depth-$k$ tree recursion, there are  tree recursions of depth $0, 1, \dots k-1$.  In particular for a depth-$k$ tree recursion $\pi_{\bl,\vtau,\vbgam}$ and a given $(v_0,\ldots,v_{j-1}) \in \cX^j$, we may ``shift'' the tree recursion to start at tuples with prefix $(v_0,\ldots,v_{j-1})$.  
 	Define the triple $(\bl',\vtau',\vbgam')$ via \begin{align}
 	\label{eq:blp-def} \bl'(y) &:= \bl(y) \vbgam(v_0, \dots, v_{j-1}, y)  \\
 	\label{eq:vtaup-def} \vtau'(y_0, \dots, y_{k-j}) &:= \vtau(v_0, \dots, v_{j-1}, y_0, \dots, y_{k-j}) \\
 	\label{eq:vbgamp-def} \vbgam'(y_0, \dots, y_{\ell}) &:=   \frac{\vbgam(v_0, \dots, v_{j-1}, y_0, \dots, y_{\ell})}{\vbgam(v_0,\ldots,v_{j-1},y_0)}\,, \text{ for } 0 \leq \ell <  k-j \,.
 	\end{align}
 	In the case that $\vbgam(v_0,\ldots,v_{j-1},y_0) = 0$, we define $\vbgam'(y_0) = 1$ and $\vbgam'(y_0,\ldots,y_\ell) = 0$ for $\ell \geq 1$.  In particular,  $\vbgam'$ is a depth-$(k-j)$ damping function and $\vtau'$ is a depth-$(k-j)$ boundary condition. 
	 With all this in place, we have the following lemma.

 	\begin{lemma}
 		\label{lemRecursiveTree}
 		Suppose $\pi_{\bl,\vtau,\vbgam } $ is a depth-$k$ tree recursion and $(v_0,\dots v_{j-1}) \in \cX^{j}$ for  some $1 \le j \leq k$.  Then for $(\bl',\vtau',\vbgam')$ as defined in equations\eqref{eq:blp-def}, \eqref{eq:vtaup-def} and \eqref{eq:vbgamp-def}  we have $$\pi_{\bl',\vtau',\vbgam'}(y_0,\ldots,y_\ell) = \pi_{\bl,\vtau,\vbgam}(v_0,\ldots,v_{j-1},y_0,\ldots,y_{\ell}) $$
 		for all  $0 \leq \ell \leq k - j$ and all $(y_0,\ldots,y_{\ell}) \in \cX^{\ell+1}$  for which the right-hand-side is defined. 
 	\end{lemma}
 	\begin{proof}
 		We induct on $k - j - \ell$ and note that the $\ell = 0$ case follows from applying \eqref{eqDepthk1}  and the definition of $\vtau'$.  We now suppose that we have proven the lemma for some $\ell$ and want to show it holds for $\ell - 1$.  Applying \eqref{eqDepthk2} shows \begin{align}
 		\pi_{\bl',\vtau',\vbgam'}(y_0,\ldots,y_{\ell-1}) &= \bl'(y_{\ell-1})\vbgam'(y_0,\ldots,y_{\ell-1}) \nonumber\\
 		&\quad \times \exp\left(- \int_{\cX} (1 - e^{-\phi(y_{\ell-1},w)})\pi_{\bl',\vtau',\bgam'}(y_0,\ldots,y_{\ell-1},w)\,d\mu(w) \right)\,. \label{eq:piprime}
 		\end{align}
 		Write  \begin{equation}\label{eq:blp-bgp}
 		\bl'(y_{\ell-1}) \vbgam'(y_0,\ldots,y_{\ell-1}) = \bl(y_{\ell-1}) \vbgam(v_0,\ldots,v_{j-1},y_0,\ldots,y_{\ell-1})\,.
 		\end{equation}
 		Applying the inductive hypothesis and combining \eqref{eq:piprime} and \eqref{eq:blp-bgp} along with  \eqref{eqDepthk2} for $\pi_{\bl,\vtau,\vbgam}(v_0,\ldots,v_{j-1},y_0,\ldots,y_{\ell-1})$ completes the inductive step.
 \end{proof}

We now show that for every $k \ge0$,  the density function $\rho_{\nu}$ of a Gibbs point process $\nu$ (defined with respect to the activity function $\bl$) can be expressed via a depth-$k$ tree recursion of a special type: the damping function $\vbgam$ is explicit and does not depend on $\bl$; moreover, the damping function of the depth-$(k+1)$ tree recursion is an extension of  the damping function of the depth-$k$ recursion (it agrees up to tuples of size $k$).  Define the damping function
\begin{equation}\label{eq:vbl-def}
	\vbgamw(v_0,\ldots,v_j) =  \prod_{i = 0}^{j-2} \exp\left( - \one_{\dist(v_j,v_{i}) < \dist(v_{i+1},v_{i})} \cdot \phi(v_j,  v_{i}) \right) \,,
	\end{equation}
where we interpret the empty product as $1$, so that $\vbgamw(v) =1$ and $\vbgamw(u,v)=1$ for all $u,v \in \cX$.  Note that since $\phi \geq 0$, this function satisfies~\eqref{eqVBsubmult}.

Now we define a specific  boundary condition $\vtauw$.   For $(v_0, \dots, v_k) \in \cX^{k+1}$, define the Gibbs measure $\nu_{v_0, \dots, v_k}$ from $\nu$ via Lemma~\ref{fact:Gibbs} using the function
\begin{equation}
\label{eqNuF}
f_{v_0,\ldots,v_{k}}(s) =\prod_{i = 0}^{k-1} \exp\left( - \one_{\dist(s,v_{i}) < \dist(v_{i+1},v_{i})} \cdot \phi(s, v_{i}) \right) \,.
\end{equation}
In particular, for $k=0$ we have $f_{v_0} \equiv 1$ and so $\nu_{v_0} = \nu$.  We also have $f_{v_0,v_1}(s) = \exp\left( - \one_{\dist(s,v_{0}) < \dist(v_1,v_{0})} \cdot \phi(s, v_{0}) \right)$, and so $\nu_{v_0,v_1} = \nu_{v_0 \to v_1}$, the Gibbs measure from Proposition~\ref{pr:infinite-vol-recursion}.

With these Gibbs measures defined we  let
 \begin{equation}\label{eq:vtau-def}
	\vtauw(v_0,\ldots,v_k) = \rho_{\nu_{v_0,\ldots,v_{k}}}(v_k)\,.
	\end{equation}
	We now show that the finite-depth tree recursion with damping function $\vbgamw$ and boundary condition $\vtauw$ computes the density of the Gibbs point process.  Further, since we have assumed that $\phi(v,\bullet)$ is measurable for all $v$, we see that in fact $\pi_{\bl,\vtau_w,\vbgam_w}$ is defined for \emph{all} tuples $(v_0,\ldots,v_j)$ with $j \leq k$ rather than just almost all tuples.

\begin{lemma}
\label{lemRhoasTree}
Consider a Gibbs measure $\nu$ with activity function $\bl$.  For $k \ge 0$ define the damping function $\vbgamw$ via~\eqref{eq:vbl-def} and the boundary conditions $\vtauw$ via~\eqref{eq:vtau-def}.  Then the resulting depth-$k$ tree recursion $\pi_{\bl,\vtau_w,\vbgam_w}$ computes the density function of $\nu$:
\[ \pi_{\bl,\vtauw,\vbgamw}(v) = \rho_{\nu}(v) \quad \text{ for all } v \in \cX   \,.\]
\end{lemma}		
	\begin{proof}
We prove this by induction on $k$.  For $k=0$, we have 
\begin{align*}
\pi_{\bl,\vtauw,\vbgamw}(v) &= \vtau(v) = \rho_{\nu_{v}}(v) = \rho_{\nu}(v) \,. 
\end{align*}

Now for $k\ge 1$, apply Proposition \ref{pr:infinite-vol-recursion} to obtain
 \begin{equation} \label{eq:induct-application} \rho_{\nu}(v_0) = \bl(v_0) \exp\left(-\int_{\cX}(1 - e^{-\phi(v_0,v_1)}) \rho_{\nu_{v_0 \to v_1}}(v_1)\,d\mu(v_1) \right)\,.
 \end{equation}
	By \eqref{eqDepthk2}, it is sufficient to prove $\rho_{\nu_{v_0 \to v_1}}(v_1) = \pi_{\bl,\vtau_w,\vbgam_w}(v_0,v_1)$. 
By the inductive hypothesis we have $\rho_{\bl_{v_0 \to v_1},\vtau_{w,v_0,v_1},\vbgam_w}(v_1) = \rho_{\nu_{v_0 \to v_1}}(v_1),$ where we write $\vtau_{w,v_0,v_1}$ to be the boundary condition obtained from \eqref{eq:vtau-def} for $\nu_{v_0 \to v_1}$.  We apply Lemma~\ref{lemRecursiveTree} for $j = 1$ to see $\pi_{\bl,\vtau_w,\vbgam_w}(v_0,v_1) = \pi_{\bl',\vtau_w',\vbgam_w'}(v_1)$ where we define $\bl', \vtau_w'$ and $\vbgam_w'$ by \eqref{eq:blp-def}, \eqref{eq:vtaup-def} and \eqref{eq:vbgamp-def}.  We first note that 
\begin{align*}
\vtau_w'(v_1,\ldots,v_k) = \vtau_w(v_0,\ldots,v_k) = \rho_{\nu_{v_0,\ldots,v_k}}(v_k)
\end{align*}
and $$\vtau_{w,v_0,v_1}(v_1,\ldots,v_k) = \rho_{(\nu_{v_0 \to v_1})_{v_1,\ldots,v_k}  }(v_k) = \rho_{\nu_{v_0,\ldots,v_k}}(v_k) $$
thus showing that the boundary conditions match.  We note that $\bl = \bl'$ since $\vbgam_w(v_0,v_1) = 1$.  Finally, we see that for $j \geq 2$ we have \begin{align}\bl_{v_0 \to v_1}(v_j)\vbgam_w(v_1,\ldots,v_j) = \vbgam_w(v_0,\ldots,v_j) \bl(v_j)\,.
\end{align}
Thus \begin{equation}\label{eq:induct-2}
\pi_{\bl',\vtau_w',\vbgam_w'}(v_1) = \pi_{\bl_{v_0 \to v_1}, \vtau_{w,v_0,v_1},\vbgamw}(v_1) = \rho_{\nu_{v_0 \to v_1}}
\end{equation}
Combining \eqref{eq:induct-application} and \eqref{eq:induct-2} with Lemma~\ref{lemRecursiveTree} and \eqref{eqDepthk2}, we obtain
\begin{align*}
\rho_\nu(v_0) &= \bl(v_0)\left(-\int_{\cX} (1 - e^{-\phi(v_0,v_1)}) \rho_{\nu_{v_0 \to v_1}}(v_1)\,d\mu \right) \\
&= \bl(v_0)\left(-\int_{\cX} (1 - e^{-\phi(v_0,v_1)}) \pi_{\bl',\vtau_w',\vbgam_w'}(v_1)\,d\mu \right) \\ 
&= \bl(v_0)\left(-\int_{\cX} (1 - e^{-\phi(v_0,v_1)}) \pi_{\bl,\vtau_w,\vbgam_w}(v_0,v_1)\,d\mu \right) \\
&= \pi_{\bl,\vtau_w,\vbgam_w}(v_0)
\end{align*}
thus proving the lemma.
	\end{proof}

\subsection{Infinite-depth tree recursions}
\label{secInfinite}

In this section we define  infinite-depth tree recursions.  These will be defined in terms of the space $(\cX, \mu,d)$, the potential $\phi$, and activity function $\bl$, and an infinite-depth damping function $\vbgam$.

A \textit{damping function} (or infinite-depth damping function) is a measurable function $\vbgam: \bigcup_{j=1}^\infty \cX^j \to [0,1]$  satisfying $\vbgam(v) =1$ for all $v \in \cX$ and satisfying property~\eqref{eqVBsubmult} for all $\ell \ge 1$ and all $(v_0, \dots, v_\ell) \in \cX^{\ell+1}$.   

Given an activity function $\bl$ and damping function $\vbgam$, an \textit{infinite-depth tree recursion} adapted to $\bl$, $\vbgam$ is a measurable function $\pi: \bigcup_{k=1}^\infty \cX^k \to [0, \infty)$ such that for all  $j \ge 0$ and almost all $(v_0, \dots, v_j) \in \cX^{j+1}$ we have 
\begin{equation}
\label{eqRecursionIdentity2}
\pi(v_0, \dots, v_{j} ) = \bl(v_{j}) \cdot \vbgam(v_0, \dots , v_{j}) \exp \left(   - \int_{\cX} (1-e^{-\phi( v_j,w )}) \pi(v_0,\ldots,v_{j},w) \, d\mu(w) \right) \,. 
\end{equation}

Extending the construction of Lemma~\ref{lemRhoasTree}, we show that there exists an infinite-depth tree recursion computing  the density of any Gibbs measure, with damping function $\vbgam$ that does not depend on $\bl$.
\begin{lemma}\label{lem:density-as-infinite-tree}
Let $\nu$ be a Gibbs measure  with activity function $\bl$.  Define the damping function $\vbgamw$ via~\eqref{eq:vbl-def} for $j \ge1$.  Then there is an infinite-depth tree recursion $\pi_w$ adapted to $\bl$, $\vbgamw$ with the property that $\pi_w(v) = \rho_{\nu}(v) $ for almost all $v \in \cX$.  Moreover, $\pi_w$ is given explicitly as 
\[ \pi_w(v_0, \dots, v_k) =   \rho_{\nu_{v_0,\ldots,v_{k}}}(v_k) \]
with $\nu_{v_0,\ldots,v_{k}}$ defined via~\eqref{eqNuF}. 
\end{lemma}
\begin{proof}
A function $\pi_w: \bigcup_{k=1}^\infty \cX^k \to [0, \infty)$ is adapted to $\bl$, $\vbgamw$ if its restriction to $\bigcup_{j=1}^{k+1} \cX^j$ is a depth-$k$ tree recursion for every $k$.  Thus the lemma follows immediately from Lemma~\ref{lemRhoasTree}.
\end{proof}

Now we can prove Theorem~\ref{thmMainTree}, which we restate now in slightly greater generality. 
\begin{theorem}
\label{thmMainTree2}
Fix a repulsive, tempered potential $\phi$ and a space $(\cX, \mu, d)$ satisfying Assumption~\ref{assumption:continuity}.   Let $\bl$ be an activity function bounded by some $\lam \ge0$. 
  Suppose that for every damping function $\vbgam \le \vbgamw$ there is at most one infinite-depth tree recursion   adapted to $\bl$, $\vbgam$.  Then there is a unique Gibbs measure on $(\cX,\mu,d)$ with potential $\phi$ and activity function $\bl$. 
\end{theorem}
\begin{proof}
Since $\phi$ is repulsive and $\bl$ is bounded by $\lam$, the $\ell$-point density functions satisfy the Ruelle bound:
\[ \rho_{\nu} (v_1, \dots, v_\ell) \le \lam^\ell  \]
for every Gibbs measure $\nu$ and all $(v_1,\dots, v_\ell) \in \cX^\ell$.   By Lemma \ref{lemRuelleBound}, $\nu$ is determined by the collection of $\ell$-point density functions $\rho_{\nu}(v_1,\ldots,v_\ell)$.  Thus, it is sufficient to show that for any two Gibbs measures $\nu$ and $\nu'$ associated to $\bl$, we have $\rho_{\nu}(v_1,\ldots,v_\ell) = \rho_{\nu'}(v_1,\ldots,v_\ell)$ for all $\ell \geq 1$ and all $(v_1, \dots, v_\ell) \in \cX^\ell$.

By Lemma~\ref{lemkPointProd}, the $\ell$-point density can be written as the product of $1$-point densities with modified activity functions, so it is sufficient to show that $\rho_{\nu_j}(v_j) = \rho_{\nu_j'}(v_j)$ where $\nu_j$ and $\nu_j'$ are as defined in Lemma \ref{lemkPointProd}; set $\bl_j$ to be the activity defined in Lemma \ref{lemkPointProd} as well.  By Lemma \ref{lem:density-as-infinite-tree}, we may find infinite-depth tree recursions $\pi_w$ and $\pi_w'$ adapted to $\bl_j, \vbgam_w$ so that $$\pi_w(v_j) = \rho_{\nu_j}(v_j) \text{ and }\pi_w'(v_j) = \rho_{\nu_j'}(v_j)\,.$$ 

Further, if we set $\vbgam_j(w_0,w_1,\ldots,w_k) = \vbgam_w(w_0,w_1,\ldots,w_k) \cdot \bl_j(w_k)/\bl(w_k)$ for $k\ge 1$, then $\vbgam_j$ is a damping function and both functions $\pi_w$ and $\pi_w'$ are  infinite-depth tree recursions  adapted to $\bl$ and $\vbgam_j$.  By assumption, such infinite-depth tree recursions are unique and so $\pi_w(v_j) = \pi_w'(v_j)$, i.e.\ $\rho_{\nu_j}(v_j) = \rho_{\nu_j'}(v_j)$.  This shows uniqueness of the $\ell$-point density functions and in turn shows uniqueness of Gibbs measure.  
\end{proof}

	\section{Contraction and convergence of densities}
	\label{secContraction}
	
	In this section we show the recursive computation above satisfies a contractive property when the activity function $\bl$ is bounded by $\lam < e/\Delta_{\phi}(\vbgam)$.  This contractive property implies that the finite-depth tree recursion  exhibits decay of dependence on the boundary condition $\vec \tau$ and that there is at most one infinite-depth tree recursion, thus proving Theorem \ref{thmUniqueTree}.   As shown in Section \ref{secTrees}, Theorem \ref{thmUniqueTree} implies  Theorem~\ref{thmMainTree2}.

	\subsection{Contraction}
	We will prove Theorem \ref{thmUniqueTree}  by the method of contraction.  Ultimately, a depth-$k$ tree recursion consists of iterating a given map.  In particular, given a bounded measurable function $\brho:\cX  \to [0,\infty)$, $v \in \cX$ and $\lambda \geq 0$, define the function
	\begin{equation}\label{eq:F-def}
	F_v(\lambda,\brho) := \lambda \exp\left(-\int_{\cX} (1 - e^{- \phi(v,w)}) \brho(w) \,d\mu(w) \right)\,.\end{equation}  
	
	A depth-$k$ tree recursion consists of a $k$-fold iteration of this function $F$, where we alter $v$ depending on the coordinates of the tuple in the tree and set $\lambda$ to be $\bl(v)$ times a corresponding damping function $\vbgam$.  In order to show uniqueness of infinite-depth tree recursions, we will show that for sufficiently large $k$, depth-$k$ tree recursions $\pi_{\bl,\vtau,\vbgam}$ satisfy a contractive property in $\vtau$, provided we are in the regime specified in Theorem \ref{thmUniqueTree}.   
	
	In order to prove this, we make use of a change of coordinates under which $F_v$ will be simpler to analyze.  This use of a change of coordinates---also called  a ``potential function''---is inspired by the contraction technique in the computer science literature~\cite{restrepo2013improved,sinclair2013spatial,sinclair2017spatial,peters2019conjecture}.  The current authors obtained a zero-free region for repulsive Gibbs point processes  in~\cite{michelen2020analyticity} using a potential function and contraction; here, we require a different change of coordinates, and must unwrap many layers of the tree-recursion at once so that we can see the geometry of the space, as measured by $\Delta_{\phi} (\vbgam)$.   We prove the following technical lemma, from which Theorem \ref{thmUniqueTree} will  follow.
	\begin{lemma}\label{lem:real-contraction}
		Consider two depth-$k$ tree recursions with common activity function $\bl$ and damping function $\vbgam$ and possibly different boundary conditions $\vtau_1$ and $\vtau_2$ respectively.  If $\bl \in [0,\lambda]$ then for each $v \in \cX$ we have $$\left|\sqrt{\pi_{\bl,\vtau_1,\vbgam}(v)} -\sqrt{\pi_{\bl,\vtau_2,\vbgam}(v)}\right|^2 \leq (\lambda/e)^{k} \cdot V_k(\vbgam) \cdot  \| \vtau_1 - \vtau_2\|_{\infty}\,.$$  
	\end{lemma} 
	
	This contraction statement---applicable only for \emph{real-valued} $\bl$---is strong enough to prove uniqueness of Gibbs measures, but a bit more will be needed in order to prove analyticity.  The analogous statement for complex activity functions is Lemma \ref{lem:complex-contraction} below.

	We begin by defining the change of coordinates with which we will prove a contraction. Define $g_v(\lambda,\bz) = \sqrt{F_v(\lambda,\bz^2)}$. 	It is not necessarily the case that $g_v$ itself is a contraction, but rather that an iterated version will be contractive when we consider the appropriate values of  $\lambda$ modulated by $\vbgam$.  Intuitively, this is because a single iteration of $g_v$ cannot ``see'' the large-scale behavior of the damping function $\vbgam$ but higher iterations of $g_v$  see more and more.  
	
	Our first step towards establishing this contraction is to apply the mean value theorem to understand a single iteration of $g_v$.

\begin{lemma}\label{lem:real-contraction-1-step}
	For any two non-negative functions $\bx ,\by$ on $\cX$, $\lambda \geq 0$, and $v \in \cX$ we have $$|g_{v}(\lambda,\bx) - g_{v}(\lambda,\by)|^2 \leq e^{-1} \lambda \int_{\cX} (1 - e^{-\phi(v,w)})|\by(w) - \bx(w)|^2 \,d\mu(w) \,.$$
\end{lemma}
\begin{proof}
	Set $\bz_t = (1 - t)\bx + t \by$ and write $\alpha_w = 1 - e^{-\phi(v,w)}$.  Then \begin{align*}
	\frac{d}{dt} g_v(\lambda,\bz_t) &= \frac{d}{dt}\left(\lambda\exp\left(-\int_{\cX} ((1 - t)\bx - t \by  )^2\alpha_w  \,d\mu(w) \right)   \right)^{1/2} \\
	&= -\left(\lambda\exp\left(-\int_{\cX}  \bz_t(w)^2 \alpha_w\,d\mu(w) \right)   \right)^{1/2} \int_{\cX}\bz_t(w) (\by(w) - \bx(w)) \alpha_w  \,d\mu(w)\,.
	\end{align*}
	By the mean value theorem we thus have \begin{equation}\label{eq:MVT}
	|g_{v}(\lambda,\bx) - g_{v}(\lambda,\by)| \leq \left(\lambda\exp\left(-\int_{\cX}  \bz_t(w)^2\, \alpha_w d\mu(w) \right)   \right)^{1/2} \int_{\cX}  \bz_t(w) |\by(w) - \bx(w)|\alpha_w \,d\mu(w) \,.
	\end{equation}
	
	Applying the Cauchy-Schwarz inequality with the measure $\alpha_w \,d\mu(w)$ gives \begin{equation}\label{eq:CS}
	\int_{\cX}\bz_t(w) |\by(w) - \bx(w)|  \alpha_w \,d\mu(w) \leq \left( \int_{\cX}  \bz_t^2(w)\alpha_w\,d\mu(w) \int_{\cX}  |\by(w) - \bx(w)|^2\alpha_w \,d\mu(w)  \right)^{1/2}\,.\end{equation}
	If we set $s :=\int_{\cX} \bz_t^2(w)\alpha_w \,d\mu(w)$ then combining lines \eqref{eq:MVT} and \eqref{eq:CS}  gives \begin{equation}
	|g_{v}(\lambda,\bx) - g_{v}(\lambda,\by)|^2 \leq \lambda e^{- s } s \int_{\cX}  |\by(w) - \bx(w)|^2 \alpha_w\,d\mu(w)\,.
	\end{equation}
	Noting that $e^{-s} s \leq 1/e$ for $s \geq 0$ completes the proof.		
\end{proof}
	
	We now iterate this bound:
	
\begin{lemma}\label{lem:real-contraction-k-step}
	In the context of Lemma \ref{lem:real-contraction}, we have \begin{align*}\left|\sqrt{\pi_{\bl,\vtau_1,\vbgam}(v_0)} -\sqrt{\pi_{\bl,\vtau_2,\vbgam}(v_0)}   \right|^2 &\leq e^{-k} \int_{\cX^k} \prod_{j = 1}^k \bl(v_j) \vbgam(v_0,\ldots,v_j) (1 - e^{-\phi(v_j,v_0)}) \\
	&\quad \times  |\sqrt{\vtau_1(v_0,\ldots,v_{k})} - \sqrt{\vtau_2(v_0,\ldots,v_{k})}|^2 d\mu^{(k)}(v_1,\ldots,v_k)\,. \end{align*}
\end{lemma}
\begin{proof}
	We proceed by induction and note that the $k = 1$ case follows immediately from Lemma \ref{lem:real-contraction-1-step}.  To complete the inductive step, for a given $v \in \cX$, we define depth-$k-1$ recursions by setting $\vbgam_v(v_1,\ldots,v_\ell) = \vbgam(v,v_1,\ldots,v_\ell)$ and for $j \in \{1,2\}$ set $\vtau_{j,v}(v_1,\ldots,v_k) = \vtau_j(v,v_1,\ldots,v_k)$.  Then by Lemma \ref{lemRecursiveTree} we have $\pi_{\bl,\vtau_1,\vbgam}(v) = F_v(\bl(v), \pi_{\bl,\vtau_{1,v},\vbgam_{v}}(\bullet) )$ and similarly for $\vtau_2$.  Thus  \begin{align*}&\left|\sqrt{\pi_{\bl,\vtau_1,\vbgam}(v_0)} -\sqrt{\pi_{\bl,\vtau_2,\vbgam}(v_0)}   \right|^2 \\
	&\qquad = \left|g_{v_0}\left(\bl(v_0),\sqrt{\pi_{\bl,\vtau_{1,v_0},\vbgam_{v_0}}(\bullet)} \right) - g_{v_0}\left(\bl(v_0),\sqrt{\pi_{\bl,\vtau_{2,v_0},\vbgam_{v_0}}(\bullet)} \right)\right|^2\,. 
	\end{align*}
	Applying Lemma \ref{lem:real-contraction-1-step} and the inductive hypothesis completes the bound.
\end{proof}
	
	\begin{proof}[Proof of Lemma \ref{lem:real-contraction}]
		Apply Lemma \ref{lem:real-contraction-k-step}, bound $\bl(v_j) \leq \lambda$ and use the elementary bound $$|\sqrt{\vtau_1(v_0,\ldots,v_k)} - \sqrt{\vtau_2(v_0,\ldots,v_k)}|^2 \leq \| \vtau_1 - \vtau_2\|_{\infty}$$ to obtain \begin{align*}\bigg|\sqrt{\pi_{\bl,\vtau_1,\vbgam}(v_0)}& -\sqrt{\pi_{\bl,\vtau_2,\vbgam}(v_0)} \bigg|^2 \\
		&\leq (\lambda/e)^k \cdot \| \vtau_1 - \vtau_2 \|_\infty  \int_{\cX^k} \prod_{j = 1}^k \vbgam(v_0,\ldots,v_j) (1 - e^{-\phi(v_j,v_0)}) \,d\mu^{(k)}(v_1,\ldots,v_k) \\
		&\leq (\lambda/e)^k \cdot \| \vtau_1 - \vtau_2 \|_\infty  \cdot V_k(\vbgam)\,.
		 \end{align*}
	\end{proof}
	
	\begin{proof}[Proof of Theorem \ref{thmUniqueTree}]
		By shifting the tree recursion using Lemma \ref{lemRecursiveTree}, we note that it is sufficient to prove uniqueness of $\pi(v)$ for each $v \in \cX$.  Let $\lam < e/\Delta_\phi(\vbgam)$ and find $\eps > 0$ so that $\lam(1 + \eps) \leq e/\Delta_\phi(\vbgam)$.  By the definition of $\Delta_\phi(\vbgam)$, we may find $k_0$ so that for all $k \geq k_0$ we have $V_k(\vbgam)^{1/k} \leq (1 + \eps/2)\Delta_{\phi}(\vbgam)$.  We note that for each $k$, we may truncate an infinite-depth tree recursion to a depth-$k$ tree recursion, and that the boundary condition is uniformly bounded by $\lambda$.  Thus, if $\pi$ and $\pi'$ are two such infinite-depth tree recursions then we may apply Lemma \ref{lem:real-contraction} for each $k \geq k_0$ to see \begin{align*}
		|\sqrt{\pi(v)} - \sqrt{\pi'(v)}|^2 &\leq \lam (\lam/e)^k V_k(\vbgam) \leq \lam  ((1 + \eps) \Delta_\phi(\vbgam))^{-k}((1 + \eps/2)\Delta_\phi(\vbgam))^k \\
		&= \lambda \left(\frac{1 + \eps/2}{1 + \eps}\right)^k\,.
		\end{align*}
		Sending $k \to \infty$ shows $\pi(v) = \pi'(v)$, thus completing the proof.
	\end{proof}
	
	\subsection{Non-uniqueness of tree recursions for large $\lambda$}
	
	If the underlying space $\cX$ is well-behaved, then even more can be said.  We say that $(\cX, \mu, d,\phi)$ is \emph{homogeneous} if for any pair of points $x,y \in \cX$ there is a bijection $g:\cX \to \cX$ with $g(x) = y$ so that $g$ preserves the metric $d$, the measure $\mu$, and the potential $\phi$.  Informally, $\cX$ is homogeneous if every point looks the same.  
	
	When $(\cX, \mu, d ,\phi)$ is homogeneous we in fact have that $e/C_\phi$ is the \textit{uniqueness threshold} for infinite-depth  tree recursions adapted to $(\phi,\lam)$: below the threshold is the uniqueness regime and above the threshold is the non-uniqueness regime.  This is analogous to the result of Kelly~\cite{kelly1985stochastic} determining the uniqueness threshold for the discrete hard-core model on the infinite $\Delta$-regular tree.  See also~\cite{brightwell1999graph}.

	\begin{theorem}
		\label{thmTreeNonUniqueness}
		Let $\phi$ be a repulsive, tempered potential and suppose that $(\cX, \mu, d, \phi)$ is homogeneous.   Let $\vbgam_1$ be the identically $1$ damping function. Then
		\begin{enumerate}
			\item For all $\lambda \geq 0$ there is an infinite-depth tree recursion adapted to $\lambda, \vbgam_1$.
			\item For $\lambda \in [0,e/C_\phi)$ the infinite-depth tree recursion is unique. 
			\item For $\lambda > e/C_\phi$ the infinite-depth tree recursion is not unique.
		\end{enumerate}
	\end{theorem}
	\begin{proof}
		First we show that for any $\lam \ge 0$ there exists an  infinite volume density function for  $(\phi,\lam)$.  Note that in this case $F(\lambda,z) = \lambda \exp(- z C_\phi)$.  We may take the constant function $\pi \equiv z^*$ where $z^*$ is the unique non-negative solution to $z^\ast = F(\lambda,z^\ast)$, and note that $z^\ast$ may be written in terms of the Lambert-W function.  By construction, this constitutes an infinite volume tree recursion.
		
		Uniqueness for $\lambda \in [0,e/C_\phi)$ follows from the more general Theorem \ref{thmUniqueTree}.  
		
		To show that there are multiple infinite-depth recursions when $\lambda > e/C_\phi$, we claim that there are multiple solutions to the equation $z = F(\lambda,F(\lambda,z))$ for $\lambda \in (0,\lambda)$; there are in fact exactly three solutions, but for our purposes it will be sufficient to show that there are at least three.  If we define $\alpha = \lambda C_\phi$ and $y = z/\lambda$, then $\alpha > e$ and $z \in [0,\lambda] \iff y \in [0,1]$.  Further, the equation $z = F(\lambda,F(\lambda,z))$ is equivalent to $$f_\alpha(y) := \exp(-\alpha \exp(-\alpha y)) - y = 0\,.$$
		We note that $f_\alpha(0) > 0$ and $f_\alpha(1) < 0$.  Further, $f_\alpha(1/e) = \exp(- \alpha \exp(-\alpha/e)) - 1/e$ which is positive for $\alpha > e$; this may be seen by noting that at $\alpha = e$ this value is zero and that the derivative is non-negative with respect to $\alpha$ for $\alpha \geq e$.  A similar argument shows $f_\alpha(1/\alpha) < 0$.  By the intermediate value theorem, there are thus zeros in the intervals $(0,1/\alpha), (1/\alpha,1/e)$ and $(1/e,1)$.
		
		Since there is only one solution in $[0,\lambda]$ to $z = F(\lambda,z)$, there must exist a solution $z_1^\ast$ satisfying $z_1^\ast = F(\lambda,F(\lambda,z_1^\ast))$ \emph{and} $z_1^\ast \neq F(\lambda,z_1^\ast)$.  If we set $z_2^\ast = F(\lambda,z_1^\ast)$ then we have $$z_1^\ast = F(\lambda,z_2^\ast) \qquad \text{ and } \qquad z_2^\ast = F(\lambda,z_1^\ast)\,.$$
		
		Thus, we may construct two recursions by first taking $$\pi(v_0,\ldots,v_k) = \begin{cases}
		z_1^\ast & k \text{ is odd} \\
		z_2^\ast & k \text{ is even}
		\end{cases}$$
		and then by swapping the roles of $z_1^\ast$ and $z_2^\ast$.
	\end{proof}

\section{Finite volume Gibbs measures and analyticity of the pressure}
\label{secAnalytic}

In this section we deduce  results about complex evaluations of partition functions  from the techniques of Section~\ref{secContraction}.  We begin with some preliminary results about finite-volume Gibbs measures and partition functions. 

We now allow activities $\bl$ to take values in $\C$.  If  $\bl$ is supported on a set $\Lam \subset \cX$ with $\mu(\Lam)<\infty$, we say $\bl$ is a \textit{finite-volume activity function}.  We assume for the rest of this section that all activity functions are bound and finite-volume, and we use $\Lam$ to refer to its support. 
We can define the \textit{finite-volume} Gibbs point process via its grand-canonical partition function,
\begin{equation}
\label{eqPPpartition2}
Z(\bl) = 1+ \sum_{k \ge1 } \frac{1}{k!}   \int_{\cX^k}  e^{-H(x_1, \dots ,x_k)}  \prod_{i=1}^k \bl(x_i) \, d\mu(x_1) \cdots d\mu(x_k) \,.
\end{equation}
When $\bl = \lam \cdot \one_{\Lam}$ we write $Z_{\Lam}(\lam)$ for the partition function. 
	When $\bl \geq 0$, the grand-canonical Gibbs point process (GPP) is the probability measure $\nu_{\bl}$ on $(\mathcal{N}_f,\mathfrak{N})$ defined by the requirement \begin{equation}\label{eq:measure-def}
	\int_{\mathcal{N}_f} g(X)\,d\nu_{\bl}(X) = \frac{1}{Z(\bl)}\left(\sum_{k \geq 0} \frac{1}{k!} \int_{\cX^k} \prod_{j = 1}^k \bl(x_j) e^{-H(x_1,\ldots,x_k)} g\left(\sum_{j = 1}^k \delta_{x_j} \right) \,d\bx \right)
	\end{equation} 
	for each suitable test function $g:\mathcal{N}_f \to \R$.

In finite volume, we can write density functions in terms of ratios of partition functions.  This follows immediately from the GNZ equations \eqref{eqGNZ}.  
\begin{lemma}
 The density of $v \in \cX$ at activity $\bl \geq 0$ can be written 
	\begin{equation}
	\label{eq1ptDef}
	\rho_{\bl}(v) = \bl (v)  \frac{Z (\bl e^{-\phi(v,\bullet)}  ) }{ Z(\bl)  } \,,
	\end{equation}
	 where $\bl e^{-\phi(v,\bullet)}: \cX \to \R_{\geq}$ denotes the function $x \mapsto \bl(x) e^{-\phi(v,x)}$. 
	 The $k$-point density at activity $\bl \geq 0$ can be written  
	\begin{equation}
	\label{eqKpt}
	\rho_{\bl}(v_1,\ldots,v_k) =\bl(v_1)\cdots \bl(v_k) \frac{Z(\bl e^{-  \sum_{i=1}^k\phi(v_i,\bullet)})}{Z(\bl)} e^{-H(v_1,\ldots,v_k)}\,,
	\end{equation}
	where $\bl e^{-\sum_{i=1}^k\phi(v_i,\bullet)} : \cX \to \R_{\ge}$ denotes the function $x\mapsto \bl(x) e^{-\sum_{i=1}^k\phi(v_i,\bullet)}$. 
\end{lemma}

In the case when $\bl$ takes values in $\C$ and $Z(\bl) \neq 0$, we use \eqref{eq1ptDef} and \eqref{eqKpt} to define complex-valued densities $\rho_{\bl}$.

We will need the following simple continuity lemma.
	\begin{lemma}\label{lem:UC}
		Let $M \geq 1$.  Then for any activities $\bl, \bl'$ uniformly bounded by $M$ we have $$|Z(\bl) - Z(\bl')| \leq \|\bl - \bl'\|_{L^1(\mu)}\exp(M \|\bl - \bl'\|_{L^1(\mu)})\,. $$
	\end{lemma}
	\begin{proof}
		Let $\bl$ and $\bl'$ be two such activities and set $\delta = \|\bl - \bl'\|_{L^1(\mu)}$.  Then \begin{align*}
		|Z(\bl) - Z(\bl')| &\leq \sum_{k \geq 1} \frac{1}{k!} \int_{\cX^k} \left|\prod_{i = 1}^k \bl(v_i) - \prod_{i = 1}^k \bl'(v_i)\right| \,d\mu^k(\bv) \\ 
		&\leq  \sum_{k \geq 1} \frac{1}{k!}\int_{\cX^k} M^{k-1} \sum_{j = 1}^k |\bl(v_i) - \bl'(v_i)| \, d\mu^k(\bv)  \\ 
		&= \delta \exp(M \delta)\,.
		\end{align*}
	\end{proof}

	We next need a version of Proposition \ref{pr:infinite-vol-recursion} suitable for complex-valued densities.  To obtain this, we recall a definition from \cite{michelen2020analyticity}: a  finite-volume, complex-valued activity function $\bl$ is \emph{totally zero-free} if for all measurable functions $\alpha:\cX \to [0,1]$, we have $Z(\alpha \cdot \bl) \neq 0$.  Our main tool is the following recursion:
	
	\begin{lemma}
	\label{lemMainIdentitiy}
		Let $v\in \cX$ and assume $\bl$ is totally zero-free.  Then 
		\begin{equation}
	\rho_{\bl}(v) = \bl(v) \exp\left(- \int_{\cX} \rho_{\bl_{v \to w}}(w)(1 - e^{-\phi(v,w)})\,dw   \right) 
		\end{equation}
		where $$\bl_{v\to w}(s) = \begin{cases}
		\bl(s) e^{-\phi(v,s)} &\text{ if }d(v,s) < d(v,w) \\
			\bl(s) & \text{ if }d(v,s) \geq d(v,w)
		\end{cases}\,. $$		
	\end{lemma}
	The proof is similar to that of~\cite[Theorem 8]{michelen2020analyticity} and so we provide it in Appendix \ref{appendix}.  We also will make use of another identity from \cite[Lemma 7]{michelen2020analyticity}, whose proof we also defer to Appendix \ref{appendix}.

\begin{lemma}
	\label{lemZIntegral}
		Let $y \in \cX$ and assume $\bl$ is totally zero-free.  Then 
		\begin{equation}
		\log Z(\bl) = \int_{\cX} \rho_{\hat{\bl}_x}(x)\,dx 
		\end{equation}
		where $$\hat{\bl}_x(w) = \begin{cases}
		0 & \text{if } d(w,y) < d(x,y) \\
		\bl(w) & \text{otherwise}
		\end{cases}\,.$$
	\end{lemma}
	
As a final ingredient, we need a complex analogue of Lemma \ref{lem:real-contraction}.  For $\eps > 0$ and an interval $[a,b]$ we write $\mathcal{N}_\eps([a,b]) := \{z \in \C : d(z,[a,b]) < \eps \}$ for the $\eps$-neighborhood of $[a,b]$ in $\C$.  We define a \emph{complex neighborhood} to be a bounded, simply-connected, open set.  

\begin{lemma}[Complex contraction] \label{lem:complex-contraction}
	For every $\lambda_0 \in (0,e/\Delta_\phi)$ there exists $k, \eps > 0$ and complex neighborhoods $U_1 \subset U_2$ so that $[0,\lambda_0] \subset U_1$ with $\overline{U}_1 \subset U_2$ so that the following holds.  For every depth-$k$ tree recursion with $\bl \in \mathcal{N}_\eps([0,\lambda_0])$, $\vtau \in \overline{U}_2$ and $v \in \cX$ we have $\pi_{\bl,\vtau,\vbgam_w}(v) \in U_1$.
\end{lemma}

The proof is a slightly more complicated version of the proof of Lemma \ref{lem:real-contraction}; the main change is that we must use a different potential function than $\sqrt{z}$, since $\sqrt{z}$ is not analytic at $0$.  Instead we use the function $\sqrt{\delta + z}$ for some $\delta > 0$.  This makes various aspects of the proof more technical, but does not alter it in a significant way.  For completeness, we prove Lemma \ref{lem:complex-contraction} in Appendix \ref{appendix}.

With Lemmas \ref{lemMainIdentitiy}, \ref{lemZIntegral} and \ref{lem:complex-contraction} in hand,  we are ready to deduce the bound stated in Theorem \ref{thmZboundCC}.
\begin{proof}[Proof of Theorem \ref{thmZboundCC}]
	Fix $\lambda_0 \in (0,e/\Delta_\phi)$.  Let $k, \eps, U_1$ and $U_2$ be as guaranteed by Lemma \ref{lem:complex-contraction} and set $C = \max\{|z|: z \in \overline{U}_2 \}$.  Let $\bl$ be such that $\bl(x) \in \cN(\lambda_0,\eps)$ for all $x \in \cX$.  For an activity function $\bl$, define the set $\cA_{\bl} = \{\bl' : \bl'(v) = \alpha(v) \bl(v), \alpha(v) \in [0,1]  \}$.  We will show that for all $\bl' \in \cA_{\bl}$ we have $$\rho_{\bl'}(v) \in \overline{U}_2 \text{ for all } v \in \cX\,.$$
	
	With this in mind, define $$\cA_{\ast} := \{\bl' \in \cA_{\bl} : \rho_{\bl''}(v) \in \overline{U}_2 \text{ for all }\bl''\in \cA_{\bl'}, v \in \cX   \}\,.$$
	
	For any $\bl' \in \cA_\ast$ note that for $\hat{\bl}_x'$ as in Lemma \ref{lemZIntegral} we have  $|\rho_{\hat{\bl}_x'(x)}| \leq C$ and so Lemma \ref{lemZIntegral} gives \begin{equation}\label{eq:it-bound}
	|\log Z(\bl') | \leq C \mu(\Lam)\,,
	\end{equation}
where $\Lam$ is the finite-volume support of $\bl$.

	To show $\lam \in \cA_\ast$, we will show that $\cA_\ast = \cA_{\bl}$.  We will essentially prove this by induction, leaning on the following claim:
	
	\begin{claim}\label{cl:induction}
		There exists an $h > 0$ so that if $\bl_\ast \in \cA_\ast$ then $\bl' \in \cA_\ast$ for all $\bl' \in \cA_{\bl}$ with $\|\bl' - \bl_\ast\|_{L^1(\mu)} < h$.
	\end{claim}
	\begin{proof}
		Consider $\rho_{\bl_\ast}(x)$ for an arbitrary $x \in \cX$.  By Lemma \ref{lemRhoasTree} we have $\rho_{\bl_\ast}(x) = \pi_{\bl_\ast,\vtau,\vbgam_w}(x)$ where for each $(v_1,\ldots,v_k)$ we have that $\vtau(x,v_1,\ldots,v_k) = \rho_{(\bl_\ast)_{x,v_1,\ldots,v_k} }(v_k)$ where $(\bl_\ast)_{x,v_1,\ldots,v_k} = (\bl_\ast)_f$ with $f$ as in \eqref{eqNuF}.  Since $(\bl_\ast)_{x,v_1,\ldots,v_k} \in \cA_{\bl_\ast}$, we thus have $(\bl_\ast)_{x,v_1,\ldots,v_k} \in \cA_\ast$, i.e. $\vtau \in \overline{U}_2$.  We then have that $\rho_{\bl_\ast}(x) \in U_1$.

		Since $\dist(\overline{U}_1,U_2^c) > 0$, we may apply Lemma \ref{lem:UC}, definition \eqref{eq1ptDef}, and the inductive bound \eqref{eq:it-bound} to find $h>0$ small enough so that if $\|\bl'  - \bl_\ast\|_{L^1(\mu)} < h$, then still have $\rho_{\bl}(x) \in U_2$ for all $x \in \cX$.  By definition, this shows $\bl' \in \cA_\ast$.
	\end{proof}

	To see that $\bl \in \cA_\ast$, note that $0 \in \cA_{\ast}$, and we may find a sequence of activities $0 \equiv \bl_0, \bl_1,\ldots,\bl_N = \bl$ with $\bl_j \in \cA_{\bl}$ and $\|\bl_j - \bl_{j+1}\|_{L^1(\mu)} < h$.  Applying Claim \ref{cl:induction} then shows that $\bl \in \cA_\ast$; noting \eqref{eq:it-bound} for $\bl$ completes the proof of the theorem.
\end{proof}

\begin{proof}[Proof of Corollary \ref{corAnalytic}]
	Fix $\lambda_0 \in [0,e/\Delta_\phi)$.  Let $D$ be the simply connected open set guaranteed by Theorem \ref{thmZboundCC} applied to $\lambda_0$.  We will show that the pressure is analytic in $D$.  Define $p_n(\lambda) = \frac{1}{n} \log Z_{\Lambda_n}(\lambda)$.  By applying Theorem \ref{thmZboundCC}, we have that $p_n(\lambda)$ is analytic in $D$ and satisfies $|p_n(\lambda)| \leq C$ for all $\lambda \in D$.  Further, we note that for all $\lambda \in [0,\lambda_0)$ we have that $p_n(\lambda)$ converges to the limit $p(\lambda)$.  By Vitali's convergence theorem (e.g., \cite[Theorem 6.2.8]{simon2015basic}), this assures that the limit $\lim_{n \to \infty} p_n(\lambda)$ exists for $\lambda \in D$ and that the limit is an analytic function.
\end{proof}

	\section*{Acknowledgments}
	MM supported in part by NSF grant DMS-2137623.
WP supported in part by NSF grants
	DMS-1847451 and CCF-1934915.

\appendix

\section{Complex activities} \label{appendix}

Here we specialize to finite volume and generalize to complex-valued activity functions $\bl$ to prove Lemmas~\ref{lemMainIdentitiy},~\ref{lemZIntegral}, and~\ref{lem:complex-contraction} and thus Theorem~\ref{thmZboundCC}. 

\subsection{Integral identities}

Before proving Lemmas \ref{lemMainIdentitiy} and \ref{lemZIntegral}, we recall some basic facts from measure theory.   By the Radon-Nikodym theorem, there exists a density $f_x$ for $\mu_x$ with respect to the Lebesgue measure.  For every $y \in \cX$ we may apply the disintegration theorem to find probability measures $\mu_{y,t}$ on the sets $\partial B_t(y)$ so that for any Borel measurable function $\psi$ we have \begin{equation}\label{eq:disintegration}
\int_{\cX} \psi(v)\,d\mu(v) = \int_0^\infty \int_{\partial B_t(y)} \psi(v) \,d\mu_{y,t}(v) f_y(t)\,dt\,. 
\end{equation}

This may be understood as an analogue of switching to spherical coordinates with ``origin'' at $y$.

\begin{lemma}\label{lem:product-dis}
	Let $\psi:\cX^j \to \C$ be a symmetric, bounded, integrable function.  Then for each $y \in \cX$ we have $$\int_{\cX^j} \psi(\bx)\,d\mu^j(\bx) = j \int_0^\infty \int_{\partial B_t(y)} \int_{(B_t(y)^c){j-1}} \psi(x_1,\ldots,x_{j-1},w) \,d\mu^{j-1}(\bx) \,f_y(t) d \mu_{y,t}(w)\,.$$
\end{lemma}
\begin{proof}
	By symmetry of $\psi$, write \begin{align*}
	\int_{\cX^j} \psi(\bx) \,d\mu^j(\bx) &= j \int_{\cX} \int_{\cX^{j-1}} \psi(x_1,\ldots,x_{j-1},w)\one_{d(y,w) \leq d(x_i,y)\,\forall i} \,d\mu^{j-1}(\bx)\, d\mu(w) \\
	&= j \int_{\cX} \int_{(B_{d(y,w)}(y)^c)^{j-1} } \psi(x_1,\ldots,x_{j-1},w) \,d\mu^{j-1}(\bx)\, d\mu(w)
	&\end{align*}
	where in the last line we used Assumption~\ref{assumption:continuity}.  Applying \eqref{eq:disintegration} then completes the proof.
\end{proof}

	\begin{proof}[Proof of Lemma \ref{lemMainIdentitiy}]
	Define $\bl_t(w) = \bl(w)(1 - \one_{d(w,y) \leq  t}(1 - e^{-\phi(y,w)} ))$ and compute \begin{align*}
	&\int_{\cX^k} \prod_{j = 1}^k \bl_t(v_j) e^{-H(\bv)}\,d\mu^k(\bv)\\
	&= \sum_{j = 0}^k (-1)^j\int_{\cX^{k-j}} \left(\prod_{i = 1}^{k-j} \bl(w_i)\right) \int_{B_t(y)^j} \prod_{i = 1}^j \left((1 - e^{-\phi(v_i,y)})\bl(v_i)\right) e^{-H(\bv,\bw)} d\mu^{j}(\bv)d\mu^{k-j}(\bw)\,.
	\end{align*}
	Denoting $\alpha_{v,y} := 1 - e^{-\phi(v,y)}$, rewrite the inner integral as: \begin{align*}
	\int_{B_t(y)^j}& \prod_{i = 1}^j \left(\alpha_{v_i,y}\bl(v_i)\right) e^{-H(\bv,\bw)} d\mu^{j}(\bv) \\
	&= j \int_{0}^t \int_{\partial B_s(y)} \alpha_{x,y}\bl(x)\int_{B_s(y)^{j-1}} \prod_{i = 1}^{j-1}\left(\alpha_{v_i,y}\bl(v_i)\right) e^{-H(x,\bv,\bw)} d\mu_{y,s}(x) \,d\mu^{j-1}(\bv)  \,ds\,.
	\end{align*}
	An application of the fundamental theorem of calculus then gives \begin{align*}
	\frac{d}{dt} Z(\bl_t) &= -\sum_{k \geq 0} \frac{1}{k!} \int_{\cX^k} \int_{\partial B_t(y)} \prod_{j = 1}^{k}\left(\bl_t(v_j)e^{-\phi(v_j,y)} \right)(1 - e^{-\phi(w,y)}) \bl(w) \,d\mu_{y,t}(w)\,d\mu^k(\bv) \\
	&= - \int_{\partial B_t(y)} \bl(w)Z(\bl_t e^{-\phi(w,\cdot)})\,d\mu_{y,t}(w)\,.
	\end{align*}
	
	Another application of the fundamental theorem of calculus gives \begin{align*}
	\log \rho_{\bl}(y) &= \log Z(\bl_\infty) - \log(Z(\bl_0)) \\
	&= -\int_0^\infty \int_{\partial B_t(y)} \bl(w)\frac{Z(\bl_t e^{-\phi(w,\cdot)})}{Z(\bl_t)} (1 - e^{-\phi(w,y)})\,d\mu_{y,t}(w)\,dt \\
	&= -\int_{\cX} (1 - e^{-\phi(w,y)}) \rho_{\bl_{y \to w}}(w)\,d\mu(w)\,.
	\end{align*}
\end{proof}

\begin{proof}[Proof of Lemma \ref{lemZIntegral}]
	Define $\bl_t$ via $\bl_t(w) = \one_{d(w,y) \geq t} \bl(w)$ and note that by assumption $Z(\bl_t) \neq 0$ for all $t$.  
	By Lemma \ref{lem:product-dis} we have \begin{align*}&Z(\bl_t) - 1 \\
	&= \sum_{j \geq 1} \frac{1}{j!} \int_0^\infty j \int_{\partial B_s(y)} \bl_{t}(w) \int_{(B_s(y)^c)^{j-1}} \prod_{i = 1}^{j-1}\bl_t(v_i) e^{-H(v_1,\ldots,v_{j-1},w)} d\mu^{j-1}(\bv)\,d\mu_{y,s}(w)\,f_y(s)\,ds \\
	&= \sum_{j \geq 1} \frac{1}{j!} \int_t^\infty j \int_{\partial B_s(y)} \bl(w) \int_{(B_s(y)^c)^{j-1}} \prod_{i = 1}^{j-1}\bl(v_i) e^{-H(v_1,\ldots,v_{j-1},w)} d\mu^{j-1}(\bv)\,d\mu_{y,s}(w)\,f_y(s)\,ds
	\end{align*}
	and so \begin{align*}
	\frac{d}{dt}&Z(\bl_t) \\
	&= -\sum_{j \geq 1} \frac{1}{(j-1)!} \int_{\partial B_t(y)} \bl(w) \int_{(B_t(y)^c)^{j-1}} \prod_{i = 1}^{j-1}\bl(v_i) e^{-H(v_1,\ldots,v_{j-1},w)} d\mu^{j-1}(\bv)\,d\mu_{y,t}(w)\,f_y(t) \\
	&= - \int_{\partial B_t(y)} \bl(w) \sum_{j \geq 0} \frac{1}{j!} \int_{\cX^j} \prod_{i = 1}^j (\bl_t(v_i) e^{-\phi(v_i,w)}) e^{-H(\bv)} \,d\mu^{j}(\bv) d\mu_{y,t}(w) \, f_y(t) \\
	&= - \int_{\partial B_t(y)} \bl(w) Z(\bl_t e^{-\phi(w,\cdot)}) f_y(t)\,d\mu_{y,t}(w)\,.
	\end{align*}
	
	By the fundamental theorem of calculus, we then have \begin{align*}
	- \log Z(\bl) &= \log Z(\bl_\infty) - \log Z(\bl_0) = \int_0^\infty \frac{d}{dt} \log Z(\bl_t) \,dt \\
	&= - \int_0^\infty \int_{\partial B_t(y)} \bl(w) \frac{Z(\bl_t e^{-\phi(w,\cdot)})}{Z(\bl_t)} f_y(t)\,d\mu_{y,t}(w)\,dt \\
	&=- \int_{\cX} \rho_{\hat{\bl}_x}(x)\,dx\,.
	\end{align*}
\end{proof}
	
	\subsection{Complex Contraction}
	
	The main goal of this subsection is to prove Lemma \ref{lem:complex-contraction}.  
	
	In the proof of Lemma \ref{lem:real-contraction}, we used the potential function $x \mapsto \sqrt{x}$.  Since this function is not analytic at $x = 0$, we need a slightly different potential function.  Fix $\lambda_0 \in [0,e/\Delta_\phi)$ and let $\delta > 0$  be a constant to be chosen later.  Set $\psi(x) = \sqrt{\delta + x}$ and $g_{v}(\lam,\bz) = \psi(F_v(\lambda,\psi^{-1}(\bz)))$ where we recall that $F_v$ is defined in \eqref{eq:F-def}.  Since the only physically meaningful second inputs of $F_v$ lie in $[0,\lambda_0]$, it is natural for us to consider $g_v$ applied to functions $\bz$ taking values in $I_\delta := \psi([0,\lambda_0]) = [\sqrt{\delta},\sqrt{\lambda_0 + \delta}]$.  Further, we may in fact expand the domain in the complex plane slightly: for each $\delta > 0$, we may take $\delta_1,\delta_2 > 0$ sufficiently small with respect to $\delta$ so that $g_v(\lam,\bz)$ is still defined for $\bz \in \cN_{\delta_1}(I_\delta)$ and $\lambda \in \cN_{\delta_2}([0,\lambda_0])$.
	
We will again be studying finite-depth tree recursions, although here we allow the boundary conditions $\vtau$ to take complex values.  Throughout, all boundary conditions $\vtau$ will be bounded and measurable and so the tree recursions are well defined.

	As before, our first step towards establishing this contraction is to apply the mean value theorem to understand a single iteration of $g_v$.  Throughout this section, we write $\alpha_{w,v} = 1 - e^{-\phi(w,v)}$, and all integrals are over $\cX$.  We note that $\alpha_{w,v} \geq 0$ and for each $v \in \cX$ we have $\int \alpha_{w,v} \,d\mu(w) =: C_v \leq C_\phi < \infty$, and so we may define the finite measure $\alpha_v$ on $\cX$ via $\alpha_v(S) = \int_S \alpha_{w,v}\,d\mu(w)$.  
	\begin{lemma}[Mean value theorem] \label{lem:MVT}
		For each $\lambda_0 \in (0,e/\Delta_\phi)$ and $\delta > 0$, we may take $\delta_1,\delta_2 > 0$ sufficient small so that the following holds.	For $\bx,\by \in \cN_{\delta_1}(I_\delta)$ and $v \in \cX$, there exists some $\bz \in\cN_{\delta_1}(I_\delta)$ so that for all $\lambda \in \cN_{\delta_2}([0,\lambda_0])$ we have $$|g_v(\lam,\bx) - g_v(\lam,\by)|^2 \leq \left|\frac{\lambda^2 e^{2\delta C_v} \exp\left(-2\int \bz(w)^2\, d\alpha_v(w)\right) }{\delta + \lambda e^{\delta C_v} \exp\left(-\int \bz(w)^2 \,d\alpha_v(w)\right) } \right| \int  |\bz(w)|^2\, d\alpha_v(w)  \int  |\by(w) - \bx(w)|^2 \,d\alpha_v(w)  \,.$$ 
	\end{lemma}
	\begin{proof}
		Let $\bz_t = (1 - t)\bx + t \by$ and note that $\bz_t \in \cN_{\delta_1}(I_\delta)$ by convexity.  Compute $$\frac{d}{dt} g_\lambda(\bz_t) = -\frac{ \lambda \exp(-\int  (\bz_t^2(w) - \delta)\,d\alpha_v(w)   ) }{\sqrt{\delta + \lambda \exp(-\int  (\bz_t^2(w) - \delta)\,d\alpha_v(w)   )}} \int\bz_t(w) (\by(w) - \bx(w)) \,d\alpha_v(w)\,. $$
		By the mean value theorem, there is a $t \in [0,1]$ so that we have $$g_\lambda(\by) - g_\lambda(\bx) = \frac{d}{dt} g_{\lambda}(\bz_t)\,.$$
		Applying the Cauchy-Schwarz inequality to the measure $d\alpha_v$ bounds $$\left|\int \bz_t(w)(\by(w) - \bx(w))\, d\alpha_v(w)\right|^2 \leq \int |\bz_t(w)|^2 d\alpha_v(w)  \int |\by(w) - \bx(w)|^2\,d\alpha_v(w)\,. $$
		Taking $\bz := \bz_t$ completes the proof. 
	\end{proof}
	
	We now bound a portion of the right-hand side of Lemma \ref{lem:MVT}.
	\begin{lemma}[Derivative bound]\label{lem:contraction}
		For each $\lambda_0 \in (0,e/C_\phi)$ and $\eps \in (0,1)$, there exists $\delta, \delta_1,\delta_2 > 0$ so that for all  $v \in \cX, \bz \in \cN_{\delta_1}(I_{\delta})$ and $\lambda \in \cN_{\delta_2}([0,\lambda_0])$ we have $$ \left|\frac{\lambda^2 e^{2\delta C_v} \exp\left(-2\int\bz(w)^2  \,d\alpha_v(w)\right) }{\delta + \lambda e^{\delta C_v} \exp\left(-\int  \bz(w)^2 \,d\alpha_v(w)\right) } \right| \int  |\bz(w)|^2\,d\alpha_v(w) \leq (1 + \eps) |\lambda|/e\,.$$
	\end{lemma}
	\begin{proof}
		For each $\delta > 0$, we may make $\delta_1,\delta_2$ sufficiently small so that $$\left|\delta +  \lambda e^{\delta C_v} \exp\left(-\int  \bz(w)^2 d\alpha_v(w)\right)\right| \geq \left|\lambda e^{\delta C_v} \exp\left(-\int \bz(w)^2\,d\alpha_v(w)\right)\right|\,.$$
		In particular, this implies $$\left|\frac{\lambda^2 e^{2\delta C_v} \exp\left(-2\int  \bz(w)^2\,d\alpha_v(w)\right) }{\delta + \lambda e^{\delta C_v} \exp\left(-\int \bz(w)^2  \,d\alpha_v(w)\right) } \right| \leq \left|\lambda e^{\delta C_v} \exp\left(-\int  \bz(w)^2 \,d\alpha_v(w)\right)\right|\,.$$
		
		By choosing $\delta_1$ sufficiently small relative to $\eps > 0$, we have $|\bz(w)^2 - |\bz(w)|^2 | \leq \eps/(2 C_v)$.  We then have $$\left| \exp\left(-\int  \bz(w)^2\,d\alpha_v(w)\right) \right| \leq e^{\eps/2}\exp\left(-\int  |\bz(w)|^2 \,d\alpha_v(w)\right)\,. $$
		
		Combining the last two displayed equations then gives \begin{align*}&\left|\frac{\lambda^2 e^{2\delta C_v} \exp\left(-2\int  \bz(w)^2 \,d\alpha_v(w)\right) }{\delta + \lambda e^{\delta C_v} \exp\left(-\int  \bz(w)^2 \,d\alpha_v(w)\right) } \right| \int|\bz(w)|^2 \,d\alpha_v(w)  \\
		&\qquad\leq |\lambda| e^{\delta C_v + \eps/2} \exp\left(-\int |\bz(w)|^2\,d\alpha_v(w)\right) \int  |\bz(w)|^2 \,d\alpha_v(w) \\
		&\qquad\leq e^{-1}|\lambda| e^{\delta C_v + \eps/2}
		\end{align*}
		using the elementary inequality $x e^{-x} \leq 1/e$ for $x \geq 0$. Bounding $C_v \leq C_\phi$ and choosing $\delta$ sufficiently small completes the proof.
	\end{proof}
	
	With Lemmas \ref{lem:MVT} and \ref{lem:contraction} in hand, we now iteratively apply their respective bounds.  Since the proof is the same as Lemma \ref{lem:real-contraction-k-step}, we omit the proof.
	
	\begin{lemma}\label{lem:comp-contract-iter} In the context of Lemma \ref{lem:contraction}, consider two depth-$k$ tree recursions with common (possibly complex) activity $\bl$ and damping function $\vbgam$ and different boundary conditions $\vtau_1$ and $\vtau_2$.   If $\bl \in \mathcal{N}_{\delta_2}([0,\lambda_0])$ and $\vtau_1,\vtau_2 \in \psi^{-1} \mathcal{N}_{\delta_1}(I_\delta)$ then 
		\begin{align*}\left|\psi(\pi_{\bl,\vtau_1,\vbgam}(v))-\psi(\pi_{\bl,\vtau_2,\vbgam}(v))\right|^2 		&\leq (1 + \eps)^k e^{-k}\int \prod_{j} |\bl(w_j)\vbgam(v,w_1,,\ldots,w_{j})| \left(1 - e^{-\phi(w_j,w_{j-1})} \right)   \\
		&\quad \times \left| \psi(\tau_1(v,\bw)) - \psi(\vtau_2(v,\bw)) \right|^2 \,d\mu^k(\bw)\end{align*}
		where we set $w_0 = v$.
	\end{lemma}	
	Applying this for the damping function  $\vbgam_w$ provides the contraction we need; again, the proof is identical to Lemma \ref{lem:real-contraction} and so we omit the proof.	
	\begin{lemma}\label{lem:deriv-1/2}
		For each $\lambda_0 \in [0,e/\Delta_\phi)$ and $r \in (0,1)$, there exists $k$, $\delta, \delta_1,\delta_2 \in ( 0,1)$ so that for all $v \in \cX$, $\bl \in \cN([0,\lambda_0])$ and depth $k$ tree recursions with boundary conditions $\vtau_1,\vtau_2 \in \psi^{-1}\cN_{\delta_1}(I_\delta)$ with damping function $\vbgam_w$ we have $$\left|\psi(\pi_{\bl,\vtau_1,\vbgam_w}(v))-\psi(\pi_{\bl,\vtau_2,\vbgam_w}(v))\right| \leq r \| \psi(\vtau_1) - \psi(\vtau_2')\|_{\infty}\,.$$
	\end{lemma}
	We have now shown that after changing coordinates, the function $\vtau \mapsto \pi_{\bl,\vtau,\vbgam_w}(v)$ is a contraction.  Due to analyticity of $F_v(\lambda,\brho)$ in $\lambda$, we  note that for each fixed $k,v, \vtau$ and $\vbgam$, the function $\bl \mapsto \pi_{\bl,\vtau,\vbgam}(v)$ is uniformly continuous with respect to $\bl$.  Recalling that we have chosen parameters so that $\psi$ is analytic in the domains we are considering proves the following lemma.
	\begin{lemma}\label{lem:lam-real}
		For each $k$, the following holds: for each $\eps_1 > 0,\delta_1 > 0,\delta > 0$, there exists a $\delta_2 > 0$ so that for all $v\in\cX, \bl,\bl' \in \cN_{\delta_2}([0,\lambda_0])$ with $\| \bl - \bl'\|_{\infty} < \delta_2$ and $\vtau \in \psi^{-1}\cN(I_\delta)$ we have $$\left|\psi(\pi_{\bl,\vtau,\vbgam_w}(v))-\psi(\pi_{\bl',\vtau,\vbgam_w}(v))\right| \leq \eps_1\,.$$
	\end{lemma}
	We are nearly at the proof of Lemma \ref{lem:complex-contraction}; putting together the pieces will prove Lemma \ref{lem:complex-contraction} up to undoing the coordinate change performed by $\psi$.
	\begin{lemma}\label{lem:pieces}
		Fix $\lambda_0 \in [0,e/\Delta_\phi)$.  Then there exists $k, \delta, \delta_1,\delta_2,\delta_3 > 0$ with $\delta_3 \in (0,\delta_1)$ so that for all $\bl \in \cN_{\delta_2}([0,\lambda_0]),\vtau \in \psi^{-1}\cN_{\delta_1}(I_\delta)$ and $v \in \cX$ we have $\pi_{\bl,\vtau,\vbgam_w}(v) \in \psi^{-1}\cN_{\delta_3}(I_\delta)$.
	\end{lemma}
	\begin{proof}
		By Lemma \ref{lem:deriv-1/2}, we may choose $k, \delta,\delta_1,\delta_2 > 0$ so that $$\left|\psi(\pi_{\bl,\vtau,\vbgam_w}(v))-\psi(\pi_{\bl,\vtau',\vbgam_w}(v))\right| \leq \frac{1}{2} \| \psi(\vsig) - \psi(\vsig')\|_{\infty}$$ whenever $\bl \in \cN_{\delta_2}([0,\lambda_0])$ and $\vtau,\vtau' \in \psi^{-1}\cN_{\delta_1}(I_\delta)$. Define $\vtau'$ via setting $\psi(\vtau(\bw)')$ to be the closest point in $[0,\lambda_0]$ to $\psi(\vtau(\bw))$ for each $\bw$ and define $\bl'$ by setting $\bl'(w)$ to be the point in $[0,\lambda_0]$ closest to $\bl(w)$; by Lemma \ref{lem:lam-real}, we may make $\delta_2$ even smaller so that we have $$|\psi(\pi_{\bl',\vtau',\vbgam_w}(v)) - \psi(\pi_{\bl,\vtau',\vbgam_w}(v)) | \leq \delta_1/4 \,.$$
		
		Combining the previous two displayed equations shows  
		We then have \begin{align}
		|\psi(\pi_{\bl,\vtau,\vbgam_w}(v)) - \psi(\pi_{\bl',\vtau',\vbgam_w}(v)) |\leq \frac{3}{4}\delta_1\,. \label{eq:cc-last-step}
		\end{align}
		Since $\pi_{\bl',\vtau',\vbgam_w}(v)\in [0,\lambda_0]$, \eqref{eq:cc-last-step}  shows $|\psi(\pi_{\bl,\vtau,\vbgam_w}(v)) \in \psi^{-1}\cN_{\delta_3}(I_\delta)$ for $\delta_3 = 3\delta_1/4$.
	\end{proof}
	\begin{proof}[Proof of Lemma \ref{lem:complex-contraction}]
		Apply Lemma \ref{lem:pieces}, and set $\eps = \delta_2$, $U_1 = \psi^{-1}\cN_{\delta_1}(I_\delta)$ and $U_2 = \psi^{-1}\cN_{\delta_3}(I_\delta)$.  Noting that $\psi$ is conformal on $\cN_{\delta_1}(I_\delta)$ shows that these preimages are complex domains and indeed $\overline{U}_2 \subset U_1$.
	\end{proof}


\begin{thebibliography}{10}

\bibitem{alder1957phase}
B.~J. Alder and T.~E. Wainwright.
\newblock Phase transition for a hard sphere system.
\newblock {\em The Journal of Chemical Physics}, 27(5):1208--1209, 1957.

\bibitem{bernard2011two}
E.~P. Bernard and W.~Krauth.
\newblock Two-step melting in two dimensions: first-order liquid-hexatic
  transition.
\newblock {\em Physical Review Letters}, 107(15):155704, 2011.

\bibitem{betsch2021uniqueness}
S.~Betsch and G.~Last.
\newblock On the uniqueness of {G}ibbs distributions with a non-negative and
  subcritical pair potential.
\newblock {\em arXiv preprint arXiv:2108.06303}, 2021.

\bibitem{brightwell1999graph}
G.~R. Brightwell and P.~Winkler.
\newblock Graph homomorphisms and phase transitions.
\newblock {\em Journal of Combinatorial Theory, Series B}, 77(2):221--262,
  1999.

\bibitem{dereudre2019introduction}
D.~Dereudre.
\newblock Introduction to the theory of {G}ibbs point processes.
\newblock In {\em Stochastic Geometry}, pages 181--229. Springer, 2019.

\bibitem{domb1972cluster}
C.~Domb and G.~Joyce.
\newblock Cluster expansion for a polymer chain.
\newblock {\em Journal of Physics C: Solid State Physics}, 5(9):956, 1972.

\bibitem{duminil-copin}
H.~Duminil-Copin and S.~Smirnov.
\newblock The connective constant of the honeycomb lattice equals
  {$\sqrt{2+\sqrt{2}}$}.
\newblock {\em Ann. of Math. (2)}, 175(3):1653--1665, 2012.

\bibitem{fernandez2007analyticity}
R.~Fern\'andez, A.~Procacci, and B.~Scoppola.
\newblock The analyticity region of the hard sphere gas. {I}mproved bounds.
\newblock {\em J. Stat. Phys.}, 5:1139--1143, 2007.

\bibitem{georgii1997stochastic}
H.-O. Georgii and T.~K{\"u}neth.
\newblock Stochastic comparison of point random fields.
\newblock {\em Journal of Applied Probability}, 34(4):868--881, 1997.

\bibitem{godsil1981matchings}
C.~D. Godsil.
\newblock Matchings and walks in graphs.
\newblock {\em Journal of Graph Theory}, 5(3):285--297, 1981.

\bibitem{groeneveld1962two}
J.~Groeneveld.
\newblock Two theorems on classical many-particle systems.
\newblock {\em Phys. Letters}, 3, 1962.

\bibitem{hammersley1954poor}
J.~M. Hammersley and K.~W. Morton.
\newblock Poor man's {M}onte {C}arlo.
\newblock {\em Journal of the Royal Statistical Society: Series B
  (Methodological)}, 16(1):23--38, 1954.

\bibitem{helmuth2020correlation}
T.~Helmuth, W.~Perkins, and S.~Petti.
\newblock Correlation decay for hard spheres via {M}arkov chains.
\newblock {\em Annals of Applied Probability}, to appear.

\bibitem{christoph2019disagreement}
C.~Hofer-Temmel.
\newblock Disagreement percolation for the hard-sphere model.
\newblock {\em Electronic Journal of Probability}, 24, 2019.

\bibitem{hofer2019disagreement}
C.~Hofer-Temmel and P.~Houdebert.
\newblock Disagreement percolation for {G}ibbs ball models.
\newblock {\em Stochastic Processes and their Applications},
  129(10):3922--3940, 2019.

\bibitem{ioffe2010statistical}
D.~Ioffe and Y.~Velenik.
\newblock The statistical mechanics of stretched polymers.
\newblock {\em Brazilian Journal of Probability and Statistics},
  24(2):279--299, 2010.

\bibitem{jansen2019cluster}
S.~Jansen.
\newblock Cluster expansions for {G}ibbs point processes.
\newblock {\em Advances in Applied Probability}, 51(4):1129--1178, 2019.

\bibitem{kelly1985stochastic}
F.~P. Kelly.
\newblock Stochastic models of computer communication systems.
\newblock {\em Journal of the Royal Statistical Society: Series B
  (Methodological)}, 47(3):379--395, 1985.

\bibitem{last2018lectures}
G.~Last and M.~Penrose.
\newblock {\em Lectures on the Poisson process}, volume~7.
\newblock Cambridge University Press, 2018.

\bibitem{lowen2000fun}
H.~L{\"o}wen.
\newblock Fun with hard spheres.
\newblock In {\em Statistical physics and spatial statistics}, volume 554,
  pages 295--331. Springer, 2000.

\bibitem{madras2013self}
N.~Madras and G.~Slade.
\newblock {\em The self-avoiding walk}.
\newblock Springer Science \& Business Media, 2013.

\bibitem{meeron1970bounds}
E.~Meeron.
\newblock Bounds, successive approximations, and thermodynamic limits for
  distribution functions, and the question of phase transitions for classical
  systems with non-negative interactions.
\newblock {\em Physical Review Letters}, 25(3):152, 1970.

\bibitem{metropolis1953equation}
N.~Metropolis, A.~W. Rosenbluth, M.~N. Rosenbluth, A.~H. Teller, and E.~Teller.
\newblock Equation of state calculations by fast computing machines.
\newblock {\em The Journal of Chemical Physics}, 21(6):1087--1092, 1953.

\bibitem{michelen2020analyticity}
M.~Michelen and W.~Perkins.
\newblock Analyticity for classical gasses via recursion.
\newblock {\em arXiv preprint arXiv:2008.00972}, 2020.

\bibitem{muthukumar1984perturbation}
M.~Muthukumar and B.~G. Nickel.
\newblock Perturbation theory for a polymer chain with excluded volume
  interaction.
\newblock {\em The Journal of chemical physics}, 80(11):5839--5850, 1984.

\bibitem{nguyen2020convergence}
T.~X. Nguyen and R.~Fern{\'a}ndez.
\newblock Convergence of cluster and virial expansions for repulsive classical
  gases.
\newblock {\em Journal of Statistical Physics}, 179:448--484, 2020.

\bibitem{penrose1963convergence}
O.~Penrose.
\newblock Convergence of fugacity expansions for fluids and lattice gases.
\newblock {\em Journal of Mathematical Physics}, 4(10):1312--1320, 1963.

\bibitem{peters2019conjecture}
H.~Peters and G.~Regts.
\newblock On a conjecture of {S}okal concerning roots of the independence
  polynomial.
\newblock {\em The Michigan Mathematical Journal}, 68(1):33--55, 2019.

\bibitem{petersen2006riemannian}
P.~Petersen, S.~Axler, and K.~Ribet.
\newblock {\em Riemannian geometry}, volume 171.
\newblock Springer, 2006.

\bibitem{restrepo2013improved}
R.~Restrepo, J.~Shin, P.~Tetali, E.~Vigoda, and L.~Yang.
\newblock Improved mixing condition on the grid for counting and sampling
  independent sets.
\newblock {\em Probability Theory and Related Fields}, 156(1-2):75--99, 2013.

\bibitem{ruelle1963correlation}
D.~Ruelle.
\newblock Correlation functions of classical gases.
\newblock {\em Annals of Physics}, 25:109--120, 1963.

\bibitem{ruelle1999statistical}
D.~Ruelle.
\newblock {\em Statistical mechanics: Rigorous results}.
\newblock World Scientific, 1999.

\bibitem{simon2015basic}
B.~Simon.
\newblock {\em Basic complex analysis}.
\newblock American Mathematical Soc., 2015.

\bibitem{sinclair2017spatial}
A.~Sinclair, P.~Srivastava, D.~{\v{S}}tefankovi{\v{c}}, and Y.~Yin.
\newblock Spatial mixing and the connective constant: Optimal bounds.
\newblock {\em Probability Theory and Related Fields}, 168(1-2):153--197, 2017.

\bibitem{sinclair2013spatial}
A.~Sinclair, P.~Srivastava, and Y.~Yin.
\newblock Spatial mixing and approximation algorithms for graphs with bounded
  connective constant.
\newblock In {\em 2013 IEEE 54th Annual Symposium on Foundations of Computer
  Science}, pages 300--309. IEEE, 2013.

\bibitem{strauss1975model}
D.~J. Strauss.
\newblock A model for clustering.
\newblock {\em Biometrika}, 62(2):467--475, 1975.

\bibitem{Weitz}
D.~Weitz.
\newblock Counting independent sets up to the tree threshold.
\newblock In {\em Proceedings of the Thirty-Eighth Annual ACM {S}ymposium on
  Theory of {C}omputing}, STOC 2006, pages 140--149. ACM, 2006.

\end{thebibliography}
\end{document}